\documentclass[12pt,a4paper]{amsart}
\usepackage{amsfonts}
\usepackage{amssymb}
\usepackage{amsmath}
\usepackage{amsthm}
\usepackage{cite}
\usepackage{enumitem}
\usepackage{tikz}
\usepackage{subfigure}
\usepackage[colorinlistoftodos]{todonotes}

\theoremstyle{plain} 
\newtheorem{thm}{Theorem}[section]

\newtheorem{lem}[thm]{Lemma}
\newtheorem{prop}[thm]{Proposition}

\newtheorem{bem}[thm]{Remark}
\newtheorem*{ack}{Acknowledgements}

\providecommand{\ov}{\overline}
\providecommand{\un}{\underline}
\providecommand{\sm}{\setminus}
\providecommand{\N}{\mathbb{N}}
\providecommand{\R}{\mathbb{R}}

\providecommand{\eps}{\varepsilon}
\providecommand{\ov}{\overline}

\providecommand{\dr}{\,dr}
\providecommand{\ds}{\,ds}
\providecommand{\dt}{\,dt}
\providecommand{\dz}{\,dz}

\DeclareMathOperator{\sign}{sign}

\DeclareMathOperator{\ind}{ind}
\DeclareMathOperator{\pr}{pr}

\newcommand{\cC}{{\mathcal C}}
\newcommand{\cD}{{\mathcal D}}

\newcommand{\cO}{{\mathcal O}}

\newcommand{\cS}{{\mathcal S}}

\newcommand{\cZ}{{\mathcal Z}}

\newcommand{\al}{\alpha}
\newcommand{\be}{\beta}

\newcommand{\la}{\lambda}

\newcommand{\De}{\Delta}

\newcommand{\La}{\Lambda}
\newcommand{\Om}{\Omega}
\newcommand{\Si}{\Sigma}

\newcommand{\pa}{\partial}

\newenvironment{altproof}[1]
{\noindent
{\em Proof of {#1}}.}
{\nopagebreak\mbox{}\hfill $\Box$\par\addvspace{0.5cm}}

\setitemize{itemsep=+2pt}
\setenumerate{itemsep=+2pt}

\numberwithin{equation}{section}

\begin{document}

\allowdisplaybreaks

\title[Infinitely many global continua bifurcating from a single solution]{Infinitely many global continua
bifurcating from a single solution of an elliptic problem with concave-convex nonlinearity}

\author{Thomas Bartsch, Rainer Mandel}
\address{T. Bartsch \hfill\break
Mathematisches Institut, Universit\"at Giessen,  \hfill\break
D-35392 Giessen, Arndtstrasse 2, Germany}
\email{Thomas.Bartsch@math.uni-giessen.de}
\address{R. Mandel \hfill\break
Scuola Normale Superiore \hfill\break
I-56126 Pisa, Piazza dei Cavalieri 7, Italy}
\email{Rainer.Mandel@sns.it}

\subjclass[2000]{Primary: 35B32; Secondary: 34C23, 35J610, 58E07}
\keywords{global bifurcation; concave-convex nonlinearity}

\begin{abstract}

We study the bifurcation of solutions of semilinear elliptic boundary value problems of the form
\begin{equation}
\left\{
\begin{aligned}
   -\De u &= f_\la(|x|,u,|\nabla u|) &&\text{in }\Om,\\
   u &= 0 &&\text{on }\pa\Om,
\end{aligned}
\right.
\end{equation}
on an annulus $\Om\subset\R^N$, with a concave-convex nonlinearity, a special case being the nonlinearity first considered by Ambrosetti, Brezis and Cerami: $f_\la(|x|,u,|\nabla u|)=\la|u|^{q-2}u + |u|^{p-2}u$ with $1<q<2<p$. Although the trivial solution $u_0\equiv0$ is nondegenerate if $\la=0$ we prove that $(\la_0,u_0)=(0,0)$ is a bifurcation point. In fact, the bifurcation scenario is very
singular: We show that there are infinitely many global continua of radial solutions
$\cC_j^\pm\subset\R\times\cC^1(\ov\Om)$, $j\in\N_0$ which bifurcate from the trivial branch $\R\times\{0\}$ at $(\la_0,u_0)=(0,0)$ and consist of solutions having precisely $j$ nodal annuli. A detailed study of these continua shows that they accumulate at $\R_{\ge0}\times\{0\}$ so that every $(\la,0)$ with $\la\ge0$ is a bifurcation point. Moreover, adding a point at infinity to $\cC^1(\ov\Om)$ they also accumulate at $\R\times\{\infty\}$, so there is bifurcation from infinity at every $\la\in\R$.
\end{abstract}

\maketitle

\section{Introduction}\label{sec:intro}
The boundary value problem
\begin{equation}\label{Gl ABC}
\left\{
\begin{aligned}
   -\De u &= \la|u|^{q-2}u + |u|^{p-2}u &&\text{in }\Om,\\
   u &= 0 &&\text{on }\pa\Om,
\end{aligned}
\right.
\end{equation}
with $\Om\subset\R^N$ a bounded domain, $1<q<2<p$ and $\la\in\R$, has received a lot of attention since being first investigated by Ambrosetti, Brezis and Cerami in \cite{ABC_Combined_effects}. Using sub- and supersolutions it is proved in \cite{ABC_Combined_effects} that there exists $\La>0$ such that \eqref{Gl ABC} has a positive solution $\un{u}_\la$ for $0<\la\le\La$. If in addition $p<2^*=\frac{2N}{(N-2)^+}$ then solutions of \eqref{Gl ABC} correspond to critical points of the functional
\[
I_\la(u)
 = \frac12\int_\Om |\nabla u|^2 - \frac{\la}{q}\int_\Om |u|^q
   - \frac1p\int_\Om |u|^p
\]
defined on $H^1_0(\Om)$, hence variational methods apply. In that case a second positive solution $\ov{u}_\la$ exist for $0<\la\le\La$ as was shown in \cite{ABC_Combined_effects}, Theorem~2.3. Moreover, there exists $\la^*>0$ such that for every $0<\la<\la^*$ problem \eqref{Gl ABC} has infinitely many solutions $\un{u}_{\la,j}$ satisfying $I_\la(\un{u}_{\la,j})<0$, and there exist infinitely many solutions $\ov{u}_{\la,j}$ satisfying $I_\la(\ov{u}_{\la,j})>0$. In \cite{BaWi_On_an_elliptic} Bartsch and Willem showed $\la^*=\infty$ as well as $I_\la(\un{u}_{\la,j})\to 0$ and $I_\la(\ov{u}_{\la,j})\to\infty$ as $j\to\infty$. In addition they showed that the solutions $\ov{u}_{\la,j}$ also exist for $\la\le0$. Furthermore, Wang \cite{Wa_Nonlinear_boundary_value} proved that the solutions $\un{u}_{\la,j}$ not only tend to 0 energetically but also uniformly on $\Omega$. Wang even deals with more general classes of nonlinearities $f_\la(u)$ instead of $\la|u|^{q-2}u + |u|^{p-2}u$. The variational structure and the oddness of the nonlinearity, however, are essential to obtain infinitely many solutions $\un{u}_{\la,j}$ and $\ov{u}_{\la,j}$. As a consequence of these results for every $\la\ge0$ the trivial solution $(\la,0)$ is a bifurcation point and there is bifurcation from infinity at every $\la\in\R$.

A precise description of the set of solutions in the one-dimensional case $\Om=(a,b)$ for positive $\lambda$ is due to Liu \cite{Liu_Exact_number_of} and Cheng \cite{Che_On_an_open}. For $j\in\N_0$ and $0<\la<\Lambda_j$ the solutions $\un{u}_{\la,j}$ and $\ov{u}_{\la,j}$ have precisely $j$ nodes and thus exactly $j+1$ nodal intervals. These pairs of solutions exist for $0<\la<\La_j$ and form a continuous curve $\cC_j\subset\R\times\cC^1[a,b]$ which, for any $j\in\N_0$, bifurcates from the trivial solution branch at the point $(0,0)\in\R\times\cC^1[a,b]$. Notice that the curve has a unique turning point at $\la=\La_j$ where $\un{u}_{\La_j,j}=\ov{u}_{\La_j,j}$ holds. We shall show in the appendix that the sets $\cC_j$ can be continued to the range $\la\le0$, not as curves but as continua (connected sets). Schematically this may be illustrated as in Figure~1.

\medskip

  \begin{figure}[!htp] 
    \centering
    \begin{tikzpicture}[yscale=0.7, xscale=0.6]
	  \draw[->] (-10,0) -- (11,0) node[right] {$\la$};
	  \draw[->] (0,0) -- (0,10)  node[above] {$\|\cdot\|_{\cC^1(\ov\Omega)}$};
	  \draw plot[smooth, tension=0.7] coordinates {  (0,0)  (2,0.8) (3.2,2) (2,3.6) (-2,5.2) (-10,7)};
  	  \draw[dashed,thin] (3.2,0) node[below]{$\La_0$} -- (3.2,2.2);
  	  \draw plot[smooth, tension=0.7] coordinates {  (0,0) (3.1,0.9) (4.5,3.5) (0,5.7) (-10,8)};
  	  \draw[dashed,thin] (4.5,0) node[below]{$\La_1$} -- (4.7,3.3);
  	  \draw plot[smooth, tension=0.6] coordinates {  (0,0) (3,0.5) (5,1.4) (7,3.8) (3,6) (-10,9)};
  	  \draw[dashed,thin] (7,0) node[below]{$\La_2$} -- (7.05,3.6);
  	  \draw plot[smooth, tension=0.5] coordinates {  (0,0)  (3,0.3)  (7,2) (8.5,5.5)  (-10,10)};
  	  \draw[dashed,thin] (9.4,0) node[below]{$\La_3$} -- (9.4,4.6);
  	  \node at (-7,6) {$\cC_0$};
  	  \node at (-5.5,6.65) {$\cC_1$};
  	  \node at (-4.5,8.25) {$\cC_2$};
  	  \node at (-3.6,9) {$\cC_3$};
  	  \draw[dashed,thin] (10,1.5) node[right] {$\un{u}_{\lambda,3}$} -- (7.6,2.5);
  	  \draw[dashed,thin] (9,7) node[right,above] {$\ov{u}_{\lambda,3}$} -- (7.6,5.7);
 	\end{tikzpicture}
 	\caption{The solution curves in the 1D case $\Om=(a,b)$.}
 	\label{Fig continua}
\end{figure}
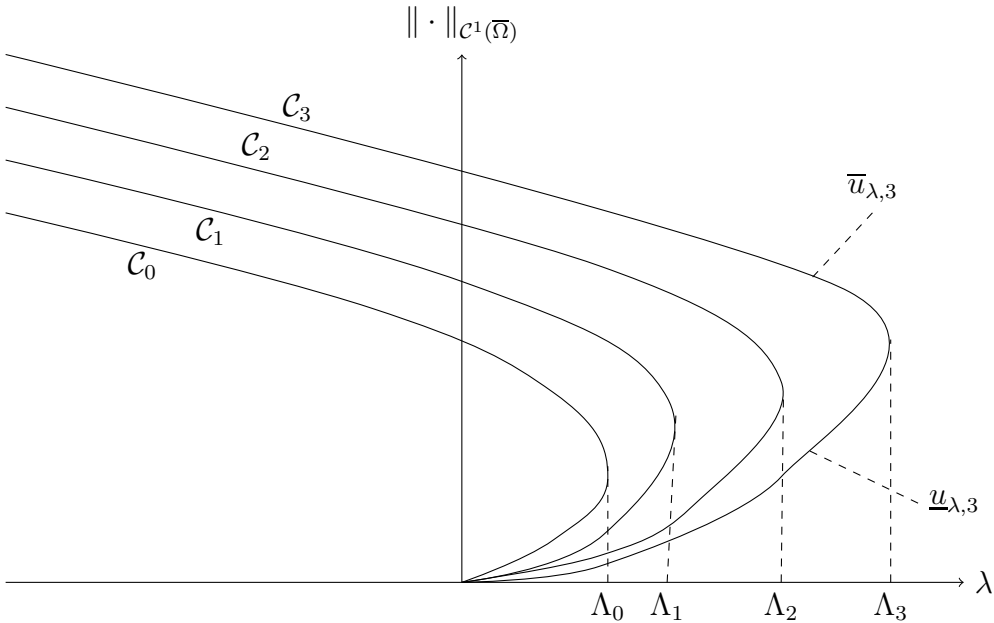

\medskip

Due to the oddness of the right hand side, there are actually two families of such curves: $\cC_j^+ = \cC_j$ and $\cC_j^- = \{(\la,-u):(\la,u)\in\cC_j\}$. In \cite{Che_On_an_open, Liu_Exact_number_of} only the case $\la>0$ has been treated. For $\la<0$ the situation becomes more complicated because there exist solutions with dead cores, that is, nontrivial solutions which vanish identically on sets of positive measure. As a consequence, the curves $\cC_j$ split and get blurred for $\la<\mu_j$ where $\mu_j<0$ can be explicitly computed. Since it is instructive and motivates the conjecture that this phenomenon occurs generically (see Remark~\ref{Bem Remark 1}) we shall give a detailed description of this phenomenon at the end of the paper. A similar behavior has been observed, for instance, in \cite{DHM09} for a quasilinear equation, and in \cite{GRS12} for a Neumann boundary problem with concave-convex nonlinearity and a parameter in the boundary condition. The papers \cite{DHM09,GRS12} deal with ordinary differential equations, but in \cite{DHM09} one can also find a discussion of the literature on dead core solutions for elliptic problems on higher dimensional domains, and on the internal free boundaries which the dead cores have.

It seems to be out of reach to obtain such detailed results for \eqref{Gl ABC} on an arbitrary bounded domain. In this paper we shall deal with a radially symmetric setting, but with a rather general nonlinearity which need not be variational nor odd in $u$. We consider the boundary value problem
\begin{equation}\label{Gl eq}
\left\{
\begin{aligned}
 - \Delta u &= f_\la(|x|,u,|\nabla u|) &&\quad\text{in }\Om, \\
	u&=0 &&\quad\text{on }\pa\Om,
\end{aligned}
\right.
\end{equation}
on the annulus $\Om:=\{ x\in\R^N : \rho_1<|x|<\rho_2\}$ with radii $\rho_2>\rho_1>0$. The nonlinearity
$f_\la(r,z,\xi)$ has a concave behavior for $z$ near $0$, and grows superlinearly for $|z|\to\infty$. The
case $f_\la(r,z,\xi)=\la|z|^{q-2}z + |z|^{p-2}z$ with $1<q<2<p<\infty$ will be covered. Our aim is to
show that there are disjoint continua $\cC_j^\pm\subset\R\times\cC^1(\Om)$ of solutions
$(\la,u)$ of \eqref{Gl eq} which bifurcate from the trivial solution $(0,0)$, that is:
$\overline{\cC_j^\pm}\cap\R\times\{0\}=\{(0,0)\}$. We shall also give a precise description of the global behaviour of the continua. They accumulate at $\R_{\ge0}\times\{0\}$ and at $\R\times\{\infty\}$, where $\infty$ is a point at infinity added to $\cC^1(\Om)$. More precisely we prove that every $(\la,0)$ with $\la\ge0$ is a bifurcation point and there is bifurcation from infinity at every $\la\in\R$. For $(\la,u)\in\cC_j^\pm$ the function $u$ is radial and has precisely
$j+1$ nodal annuli. However, in case $\lambda<0$ it may have dead cores consisting of unions of annuli.

There are a number of difficulties to overcome. Firstly, there is no uniqueness of positive or negative radial solutions of \eqref{Gl eq} in annuli which would allow to patch solutions as in \cite{BW93} or as in the proof of the result for \eqref{Gl ABC} on an interval. Secondly, the problem is in general non-variational, and the nonlinearity is not odd in $u$, hence Ljusternik-Schnirelmann theory does not apply. Thirdly, the bifurcation scenario is very singular, and results like Rabinowitz' global bifurcation theorem do not apply either. In fact, due to the concave behavior of $f_\lambda(r,z,\xi)$ for $z$ near $0$ equation \eqref{Gl eq} cannot be linearized at $z=0$, except when $\la=0$. This concave behavior of $f$ also causes problems when applying ODE techniques, and it is responsible for the existence of dead core solutions.

Here is an outline of the paper. In Section~\ref{sec:results} we will state the precise assumptions on $f$, and we will formulate and discuss our main results on the existence and shape of the continua $\cC_j^\pm$. In Section~\ref{sec:transform} we transform the ODE boundary value problem corresponding to \eqref{Gl eq} into an equivalent problem on $(0,1)$ using a suitable nonlinear transformation of the radial coordinate. The corresponding result will be formulated in Theorem~\ref{Thm 2}. In Section~\ref{sec:a-priori} we use a modification of the time map technique to prove (mostly explicit) a-priori estimates for solutions of the ODE problem depending on the number of their nodal intervals. In Section~\ref{sec:proof-thm2} these estimates
will be used to perform a degree argument in order to prove Theorem~\ref{Thm 2} which, using the corresponding inverse transformation, almost immediately provides the proof of Theorem~\ref{Thm 1}. In Appendix~A we explicitly determine the solution continua $\cC_j^\pm$ for the one-dimensional Ambrosetti-Brezis-Cerami problem in the interval $(0,1)$ and all $\la\in\R$. This includes a detailed description of the dead cores which appear for $\la<0$. For $\la\ge 0$ the existence of the solutions is not new, as mentioned above, but we include it in order to recall how exact multiplicity results can be proved with the aid of the time-map technique when the space dimension is one (so that the ODE problem is autonomous). In Appendix~B we give the proof of some technical propositions.

\section{Statement of results}\label{sec:results}
We first state our hypotheses. We assume that there are positive numbers $m_1,M_1$ and exponents $p,q$
satisfying $1<q<2<p<\infty$ such that the following holds:
\begin{itemize}
\item[(A1)]
  The map
  \[
  f:\R\times[\rho_1,\rho_2]\times\R\times\R_{\ge 0}\to\R,\quad (\la,r,z,\xi)\mapsto f_\la(r,z,\xi),
  \]
  is continuous, and it is differentiable with respect to $r,\xi$. Moreover, for all $\la,s>0$ there is a
  $K_1(\la,s)>0$ such that
  \[
  |\pa_r f_\la(r,z,\xi)|,\ |\pa_\xi f_\la(r,z,\xi)|\le K_1(\la,s)|f_\la(r,z,\xi)|
  \]
  for $r\in [\rho_1,\rho_2]$ and $0\leq |z|,\xi\le s$.
\item[(A2)] For all $\la\ge0$, $r\in[\rho_1,\rho_2]$, $z\in\R$, $\xi\in\R_{\ge0}$ we have
  \[
  m_1 (\la |z|^q+|z|^p) \le z f_\la(r,z,\xi) \le M_1 (\la|z|^q+|z|^p)
  \]
\item[(A3)] For all $\la\le0$, $r\in[\rho_1,\rho_2]$, $z\in\R$, $\xi\in\R_{\ge0}$
  we have
  \[
    M_1\la |z|^q + m_1 |z|^p \le z f_\la(r,z,\xi) \le m_1 \la |z|^q+ M_1 |z|^p.
  \]
\end{itemize}
Moreover we want to add an assumption which allows to estimate the energy of the constructed solutions in case the system is variational, i.e.\ when the right hand side in \eqref{Gl eq} does not depend on $|\nabla u|$. To this end we introduce the following condition:
\begin{itemize}
\item[(A4)] $f_\la(r,z,\xi) = f_\la(r,z)$ and the function
  $F_\la(r,z):=\int_0^z f_\la(r,s)\ds$ satisfies
  \[
  \liminf_{|z|\to\infty}\inf_{r\in[\rho_1,\rho_2]} \frac{f_\la(r,z)z}{F_\la(r,z)} > 2
  \quad\text{if } \la\in\R,\qquad
  \limsup_{|z|\to 0}\sup_{r\in [\rho_1,\rho_2]} \frac{f_\la(r,z)z}{F_\la(r,z)} < 2
  \quad\text{if } \la>0.
  \]
\end{itemize}
In this case the energy is defined by
$$
I_\la(u) := \frac12 \int_\Om |\nabla u(x)|^2\,dx - \int_\Om F_\la(|x|,u(x))\,dx.
$$
As mentioned earlier all of these conditions are satisfied for nonlinearities of Ambrosetti-Brezis-Cerami type. For instance, (A1),(A2),(A3) are satisfied if $f$ is given by 
$$
  f_\la(r,z,\xi) = \la a_q(r,\xi,\la)|z|^{q-2}z + g_\la(r,z,\xi) + a_p(r,\xi,\la)|z|^{p-2}z
$$
where $1<q<2<p<\infty$, $a_q,a_p$ are bounded continuously differentiable functions which are bounded from
below by a positive constant, and $g$ satisfies $g_\la(r,z,\xi)/\max\{|z|^{q-1},|z|^{p-1}\}\to0$ as
$|z|\to0$ or $|z|\to\infty$. Clearly, (A4) holds in that case if and only if $a_q,a_p$ and $g$ do not depend
on $\xi$.

We work on the space $X=C^1_{rad}(\ov\Om,\R)$ of radial $\cC^1$-functions $u:\ov\Om\to\R$. A solution $(\la,u)\in\R\times X$ of \eqref{Gl eq} is defined to be a solution of the integral equation associated to \eqref{Gl eq} since only the latter makes sense for merely continuously differentiable functions. Given the radial symmetry and assumption~(A1) it is immediate that every solution of the integral equation is twice continuously differentiable on $\ov\Om$ and solves the boundary value problem \eqref{Gl eq} in the classical sense. The set of solutions
\[
\cS = \{(\la,u)\in\R\times X:u\ne 0,\ (\la,u)\text{ is a classical solution of }
     \eqref{Gl eq}\}
\]
contains the sets
\[
\cS_j^\pm
 = \{(\la,u) \in \cS :  u \text{ has precisely }j+1 \text{ nodal annuli }
    A_0,\ldots,A_j, \sign(u|_{A_k})=\pm(-1)^k\},
\]
$j\in\N_0$. Here a function $u:\ov\Om\to\R$ is said to have precisely $j+1$ nodal annuli if there are mutually disjoint open annuli $A_0,\ldots,A_j\subset \Om$ such that $|u|>0$ on $A_k$ for $k=0,\ldots,j$, and $u=0$ on $\Om\setminus\bigcup_{k=0}^jA_k$. In particular, $\sign(u|_{A_k})$ is well-defined for $k=0,\ldots,j$. If the space dimension is one a nodal annulus will be called a nodal interval. Notice that we do not require $\ov\Om=\bigcup_{k=0}^j\ov{A_k}$ so that dead core solutions are permitted. Indeed, for our purposes such a requirement would be too restrictive since solutions $(\la,u)\in\cS_j^\pm$ are expected to possess dead cores if the parameter $\la$ is negative and has sufficiently large absolute value, see also 
Remark~\ref{Bem Remark 1}~(d) and Proposition~\ref{Thm 1D case}. We can now state our result.

\begin{thm}\label{Thm 1}
Let $f$ satisfy (A1), (A2), (A3) for $1<q<2<p<\infty$. Then there are maximal connected sets $\cC_j^\pm\subset\cS_j^\pm$, $j\in\N_0$, having the following properties:
\begin{itemize}
\item[(i)] $\ov{\cC_j^\pm}\setminus\cC_j^\pm = \{(0,0)\}$ for all $j\in\N_0$.
\item[(ii)] $\ov{\cC_j^\pm}\cap\ov{\cC_k^\pm} = \ov{\cC_j^+}\cap\ov{\cC_j^-} = \{(0,0)\}$ for all $j,k\in\N_0$ with $j\neq k$.
\item[(iii)] There are sequences $(\La_j^\pm)_{j\in\N_0}$ in $(0,\infty)$ tending to
infinity such  that $\pr(\cC_{j}^\pm)=(-\infty,\La_j^\pm]$. 
\item[(iv)] We have
  $\liminf_{j\to\infty}\cC_j^\pm=\limsup_{j\to\infty}\cC_j^\pm=\R_{\ge0}\times\{0\}$.
  Furthermore:
  \begin{itemize}
  \item[(a)] For all $\la>0$ the point $(\la,0)$ is a bifurcation point but not a
  branching point for \eqref{Gl eq} and there is a number $j^{\pm}(\la)\in\N_0$ such that for all $j\ge j^{\pm}(\la)$ there are solutions $(\la,\un{u}_j^\pm)\in\cC_j^\pm$ with $\|\un{u}_j^\pm\|_{\cC^1}\to 0$ as $j\to\infty$.
  \item[(b)] Any $\la\in\R$ is a bifurcation point from infinity but not a
  branching point from infinity for \eqref{Gl eq}, and setting $j^{\pm}(\la)=0$ for $\la\le 0$, for all $j\ge j^{\pm}(\la)$ there are solutions
  $(\la,\ov{u}_j^\pm)\in \cC_j^\pm$ with $\|\ov{u}_j^\pm\|_\infty\to\infty$ as $j\to\infty$.
  \end{itemize}
\item[(v)] In case (A4) holds the solutions from (iv)(a),(b) satisfy
  \[
  I_\la(\un{u}_j^\pm)\to 0^-\quad\text{and}\quad I_\la(\ov{u}_j^\pm)\to \infty \quad\text{as }j\to\infty.
  \]
\end{itemize}
\end{thm}

Before we comment on this result let us explain the notation which we used in the statement of Theorem~\ref{Thm 1}. For $\cC\subset \R\times X$ the set $\pr(\cC)$ denotes the projection of $\cC$ onto the parameter space which is formally defined by
\[
\pr(\cC):= \{\la\in\R:\text{there is } {u\in X} \text{ such that }(\la,u)\in\cC\}.
\]
The symbol $\limsup_{j\to\infty} \cC_j$ denotes the set of accumulation points of the sequence $(\cC_j)_j$ which consists of all points $(\la,u)\in \R\times X$ such that every neighbourhood of $(\la,u)$ contains elements of infinitely many $\cC_j$. Similarly, $\liminf_{j\to\infty} \cC_j$ is the set of limit points containing precisely those points such that every neighbourhood contains elements of almost all $\cC_j$. In
Theorem~\ref{Thm 1}~(iv)(a) the point $(\la,0)$ is called a bifurcation point (with respect to the family $\R\times\{0\}$ of trivial solutions) if there is a sequence $(\la_k,u_k)_k$ in $\cS$ with $(\la_k,u_k)\to(\la,0)$ and $u_k\neq 0$ as $k\to\infty$. In Theorem~\ref{Thm 1}~(iv)(b) we say that $\la$ is a bifurcation point from infinity if there is a sequence $(\la_k,u_k)_k$ in $\cS$ with $\|u_k\|_\infty\to\infty$ and
$\la_k\to\la$ as $k\to\infty$. Finally, a bifurcation point $(\la,0)$ is called a branching point if there is
a connected set in $\cS$ the closure of which contains $(\lambda,0)$. Similarly, $\la$ is said to be a
branching point from infinity if there is a connected set in $\cS$ such that there are solutions $(\mu,u)$
belonging to this connected set which satisfy $\|u\|_{\cC^1}\to \infty$ and $\mu\to\lambda$.

\begin{bem}\label{Bem Remark 1}
\begin{itemize}
\item[a)] A main feature of Theorem~\ref{Thm 1} is the fact that it proves the
  existence of infinitely many connected continua in a degenerate bifurcation setting. The results of   Ambrosetti-Brezis-Cerami \cite[Theorem~2.5]{ABC_Combined_effects}, Bartsch-Willem
  \cite[Theorem~1.1]{BaWi_On_an_elliptic} and Wang \cite[Theorem~1.1]{Wa_Nonlinear_boundary_value} are significantly improved in the special case of the annulus. It would be very interesting to investigate the case of a ball. Here for the global behavior of the continua one probably has to distinguish the cases $p<2^*$, $p=2^*$ and $p>2^*$; see Fig.~2, p.523, in \cite{ABC_Combined_effects}. In the case of a general bounded domain the existence of solution continua and the geometry of the nodal sets as in Theorem~\ref{Thm 1} remain a challenging open problem.
\item[b)] The properties of the solution continua $\cC_j^\pm$ from Theorem~\ref{Thm 1} are of qualitative nature. Aiming for a result which is strongest possible we could include the a-priori bounds for the associated ODE boundary value problem from Lemma~\ref{Lem Apriori estimates I} and Lemma~\ref{Lem Apriori estimates II} which provide further information about the localization of the solution continua. Since these estimates require the definition of several constants and mappings we decided not to include them into Theorem~\ref{Thm 1}.
\item[c)] As mentioned in the introduction the above result can be proved by explicit means when $n=1$ and $f_\la(r,z,\xi)= \la|z|^{q-2}z+|z|^{p-2}z$. We shall do this in Theorem~\ref{Thm 1D case} of Appendix~A. The proof there shows that the solutions have a dead core for $\la\ll0$.
\item[d)] The analysis of the one-dimensional case (see appendix A) leads to the
  conjecture that there is a threshold value $\un{\la}_j(\Om,p,q)<0$ such that every solution $(v,\la)\in\cS_j^\pm$ with $\la<\un{\la}_j(\Om,p,q)$ has a dead core. A thorough investigation of the formation of dead cores as $\la\to -\infty$ remains open.
\end{itemize}
\end{bem}

\section{Transforming the problem} \label{sec:transform}
Since we aim at proving the existence of radially symmetric solutions $u$ of \eqref{Gl eq} with a prescribed number of zeros it is convenient to consider the corresponding boundary value problem for the radial profile $w$ defined by the equation $u(x)=w(|x|)$ and satisfying $w(\rho_1)=w(\rho_2)=0$. This boundary value problem is given by
\begin{equation}\label{Gl ODE 2}
\left\{
\begin{aligned}
 &- w'' - \frac{N-1}{r} w' = f_\la(r,w,|w'|) \quad\text{in }(\rho_1,\rho_2),\\
 &w(\rho_1) = w(\rho_2) = 0.
\end{aligned}
\right.
\end{equation}
In Proposition \ref{Prop 1} we show that the diffeomorphism $\phi:[0,1]\to[\rho_1,\rho_2]$ given by
\begin{equation} \label{Gl Def phi}
\phi(r):=\rho_1^{1-r}\rho_2^r \quad\text{if }N=2,\qquad
\phi(r):= (\rho_1^{2-n}+r(\rho_2^{2-n}-\rho_1^{2-n}))^{1/(2-n)}
 \quad\text{if }N\ne 2,
\end{equation}
transforms \eqref{Gl ODE 2} into the boundary value problem
\begin{equation}\label{Gl eq ODE}
\left\{
\begin{aligned}
&- v''(r) = h_\la(r,v(r),|v'(r)|)\quad\text{in }(0,1),\\
&v(0)=v(1)=0,
\end{aligned}
\right.
\end{equation}
where the function $h_\la:[0,1]\times\R\times\R_{\ge0}$ is defined by
\begin{equation}\label{Gl Definition h}
  h_\la(r,z,\xi) = \phi'(r)^2f_\la(\phi(r),z,\xi/\phi'(r)).
\end{equation}

\begin{prop} \label{Prop 1}
The following holds for functions $u:\Om\to\R$ and $w:[\rho_1,\rho_2]\to\R$ related by $u(x)=w(|x|)$.
\begin{itemize}
\item[(i)] A function $u$ is a classical solution of \eqref{Gl eq} if and only if
  $w\circ\phi$ is a classical solution of \eqref{Gl eq ODE}.
\item[(ii)] We have $0<m_2\le \phi'(r)\le M_2$ for all $r\in [0,1]$ where $m_2,M_2$
  are given by
  \begin{align*}
   &m_2 = \rho_1\ln\Big(\frac{\rho_2}{\rho_1}\Big),
    &&M_2 = \rho_2\ln\Big(\frac{\rho_2}{\rho_1}\Big)
    &&\text{in case }N=2, \\
   &m_2 = \frac{\rho_1}{N-2}\Big(1-\Big(\frac{\rho_1}{\rho_2}\Big)^{N-2}\Big),
    &&M_2 = \frac{\rho_2}{N-2}\Big(\Big(\frac{\rho_2}{\rho_1}\Big)^{N-2}-1\Big)
    &&\text{in case }N\ne 2.
  \end{align*}
\end{itemize}
\end{prop}

\begin{proof}
Clearly $u$ solves \eqref{Gl eq} if, and only if, $w$ solves \eqref{Gl ODE 2}. One immediately checks that
$$
\phi'' = (\phi')^2\cdot \frac{N-1}{\phi} \quad\text{on }(0,1),\qquad
\phi(0)=\rho_1,\quad \phi(1)=\rho_2.
$$
Hence, the function $v:= w\circ \phi:[0,1]\to\R$ satisfies $v(0)=v(1)=0$ and
\begin{align*}
-v''(r)
 &= -w''(\phi(r))\phi'(r)^2-w'(\phi(r))\phi''(r)
  = \phi'(r)^2 \left(-w''(\phi(r))-\frac{N-1}{\phi(r)}w'(\phi(r)) \right) \\
 &= \phi'(r)^2 f_\la(\phi(r),w(\phi(r)),|w'(\phi(r))|)
  = h_\la(r,v(r),|v'(r)|)
\end{align*}
so that $v$ is a classical solution of \eqref{Gl eq ODE}. Similarly it can be checked that the opposite implication is true and we obtain part (i). The estimate from part (ii) follows from
\begin{align*}
 &\phi'(r) = \rho_1^{1-r}\rho_2^r\ln\Big(\frac{\rho_2}{\rho_1}\Big)
  &&\text{for }r\in [0,1]\text{ if }N=2,\\
 &\phi'(r) = \frac{\rho_1^{2-N}-\rho_2^{2-N}}{N-2}
              \Big(\rho_1^{2-N}+r(\rho_2^{2-N}-\rho_1^{2-N})\Big)^{(N-1)/(2-N)}
  &&\text{for }r\in [0,1]\text{ if }N\ne 2.
\end{align*}
\end{proof}

By Proposition \ref{Prop 1} (i) the original problem \eqref{Gl eq} is equivalent to the boundary value problem \eqref{Gl eq ODE} on the interval $(0,1)$ and we may content ourselves with proving the ODE version of Theorem \ref{Thm 1}. To this end let us fix the properties of the function $h_\la$ from \eqref{Gl Definition h} which correspond to the assumptions (A1), (A2), (A3) for the function $f_\la$. Setting
\begin{equation} \label{Gl def mMKc}
\begin{aligned}
 &m:=m_1m_2^2,\quad\qquad M:=M_1M_2^2\quad\text{and}
 \\
 &K(\la,s):= 2m_2^{-1}\|\phi''\|_\infty+(sm_2^{-2}\|\phi''\|_\infty+m_2^{-1}+M_2) K_1(\la,s/m_2)	
\end{aligned}
\end{equation}
for $m_1,M_1,K_1$ as in (A2) and $m_2,M_2,\phi$ as in Proposition \ref{Prop 1} (ii) we obtain the following:
\begin{itemize}
\item[(B1)]
  The map
  \[
  h:\R\times[0,1]\times\R\times\R_{\ge0},\quad (\la,r,z,\xi)\mapsto h_\la(r,z,\xi)
  \]
  is continuous, and it is continuously differentiable with respect to $r$ and $\xi$. Moreover, for all
  $\la,s>0$ there is a $K(\la,s)>0$ such that
  $$
    |\pa_r h_\la(r,z,\xi)|,\ |\pa_\xi h_\la(r,z,\xi)| \le K(\la,s)|h_\la(r,z,\xi)|
  $$
  for $r\in [\rho_1,\rho_2]$ and $0\leq |z|,\xi \le s$.
\item[(B2)] For all $\la\ge 0$, $r\in [0,1]$, $z\in\R$, $\xi\in\R_{\ge 0}$  we have
  $$
    m (\la |z|^q+|z|^p) \le z h_\la(r,z,\xi) \le M (\la|z|^q+|z|^p)
  $$
\item[(B3)] For all $\la\le 0$, $r\in [0,1]$, $z\in\R$, $\xi\in\R_{\ge 0}$ we have
  $$
    M\la |z|^q + m|z|^p \le z h_\la(r,z,\xi) \le m \la |z|^q+ M |z|^p.
  $$
\end{itemize}
In Theorem \ref{Thm 2} we will formulate our results concerning the boundary value problem \eqref{Gl eq ODE} for all nonlinearities $h_\la$ satisfying the assumptions (B1), (B2), (B3). As before we find a statement about the energy of the constructed solutions once we require that the equation is variational and satisfies the following condition:
\begin{itemize}
\item[(B4)] $h_\la(r,z,\xi)=h_\la(r,z)$ and the function $H_\la(r,z):=\int_0^z h_\la(r,s)\ds$ satisfies
  \begin{align*}
    \liminf_{|z|\to\infty} \inf_{r\in [0,1]} \frac{h_\la(r,z)z}{H_\la(r,z)} > 2 \quad\text{if }\la\in\R,
   \qquad
   \limsup_{|z|\to 0} \sup_{r\in [0,1]} \frac{h_\la(r,z)z}{H_\la(r,z)} < 2 \quad\text{if }\la>0.
  \end{align*}
\end{itemize}
In case (B4) holds the energy functional $J_\la:Y\to\R$ associated to \eqref{Gl eq ODE} is given by
$$
  J_\la(v) := \frac12 \int_0^1 v'(r)^2\,dr - \int_0^1 H_\la(r,v(r))\,dr
$$
where $Y:= C^1([0,1],\R)$. In the statement of Theorem \ref{Thm 2} we need the following subsets of $\R\times Y$ which are the one-dimensional analogues of the subsets $\cS,\cS_j^\pm$ of $\R\times X$:
\begin{align*}
\Si
 &= \{(\la,v)\in\R\times Y:v\ne0\text{ and }(\la,v)\text{ solves }\eqref{Gl eq ODE}\}, \\
\Si_j^\pm
&= \{(\la,v)\in\Si:v\text{ has precisely } j+1 \text{ nodal intervals } I_0,\ldots,I_j,
 \sign(v|_{I_k})=\pm (-1)^k\}.
\end{align*}
Then the analogue of Theorem \ref{Thm 1} for the boundary value problem \eqref{Gl eq ODE} then reads as follows.

\begin{thm} \label{Thm 2}
 Let $h_\la$ satisfy (B1),(B2),(B3) for $1<q<2<p<\infty$. Then there are maximal connected sets $\cD_j^\pm\subset\Si_j^\pm$, $j\in\N_0$, having the following properties:
 \begin{itemize}
   \item[(i)] $\ov{\cD_j^\pm}\setminus\cD_j^\pm = \{(0,0)\}$ for all $j\in\N_0$.
   \item[(ii)] $\ov{\cD_j^\pm}\cap \ov{\cD_k^\pm} = \ov{\cD_j^+}\cap \ov{\cD_j^-} = \{(0,0)\}$ for all $j,k\in\N_0$ with $j\ne k$.
   \item[(iii)] There are sequences $(\La_j^\pm)_{j\in\N_0}$ in $(0,\infty)$ tending to infinity such  that $\pr(\cD_{j}^\pm)=(-\infty,\La_j^\pm]$.
   \item[(iv)] We  have
       $\liminf_{j\to\infty} \cD_j^\pm = \limsup_{j\to\infty} \cD_j^\pm
        = \R_{\ge 0}\times\{0\}.$
   Furthermore:
    \begin{itemize}
    \item[(a)] For all $\la>0$ the point $(\la,0)$ is a bifurcation point but not a branching point for \eqref{Gl eq ODE}, and there is a number $j^{\pm}(\la)\in\N_0$ such that for all $j\ge j^{\pm}(\la)$ there are solutions $(\un{v}_j^\pm,\la)\in\cD_j^\pm$ with   $\|\un{v}_j^\pm\|_{\cC^1}\to 0$ as $j\to\infty$.
    \item[(b)]  Every $\la\in\R$ is a bifurcation point from infinity but not a    branching point from infinity for \eqref{Gl eq ODE}, and setting $j^\pm(\la)=0$ for $\la\le 0$, for all $j\ge j^{\pm}(\la)$ there are solutions $(\ov{v}_j^\pm,\la)\in \cD_j^\pm$ with    $\|\ov{v}_j^\pm\|_\infty\to \infty$ as $j\to\infty$.
    \end{itemize}
   \item[(v)] In case (B4) holds the solutions from (iv)(a),(b) satisfy
   $$
   J_\la(\un{v}_j^\pm)\to 0^-\quad \text{ and } J_\la(\ov{v}_j^\pm)\to \infty
      \qquad\text{as }j\to\infty.
   $$
 \end{itemize}
\end{thm}


\section{A priori estimates} \label{sec:a-priori}

In this section we prove a-priori estimates for nontrivial solutions of 
\eqref{Gl eq ODE} depending on their number of zeros. For further reference we introduce the map
\begin{equation}\label{def:g_la}
g_\la(z) = \la|z|^{q-2}z+|z|^{p-2}z \qquad\text{for }z\in\R.
\end{equation}
The first result deals with the case $\la\le0$.

\begin{lem} \label{Lem Apriori estimates I}
 Assume that (B1), (B3) hold, and let $j\in\N_0$. Then there are positive numbers $D_j$ and $d$ independent of $j$ such that all $(\la,v)\in\Si_j$ with $\la\le 0$ satisfy
 $$
   d\big( (j+1)^{\frac{2}{p-2}}+|\la|^{\frac1{p-q}} \big)
   \le \|v\|_\infty \le D_j(1+|\la|^{\frac1{p-q}})\quad\text{and}\quad
   \|v'\|_\infty \le M g_{|\la|}(\|v\|_\infty),
 $$
 where $M$ is from (B3). Moreover, for every nodal interval $I$ of $v$ we have
 $$
   \|v\|_{L^\infty(I)} \ge d\big( |I|^{-\frac{2}{p-2}} +|\la|^{\frac1{p-q}} \big).
 $$
\end{lem}

\begin{proof}
Let $(\la,v)\in\Si_j^\pm$ with $\la\le 0$ and let $I$ be a nodal interval of $v$. Multiplying the differential equation \eqref{Gl eq ODE} with $v$ and integrating the resulting equation over $I$ gives
$$
\int_I v'(r)^2\dr
 = \int_I h_\la(r,v(r),v'(r))v(r)\dr
 \le m \la \int_I |v(r)|^q \dr + M \int_I |v(r)|^p\dr.
$$
Using the fact that $\pi^2|I|^{-2}$ is the smallest Dirichlet eigenvalue of the differential operator $-\frac{d}{dr^2}$ on $I$ we obtain the estimate
\begin{align*}
\pi^2|I|^{-2} + m |\la|\|v\|_{L^\infty(I)}^{q-2}
 &\le \frac{\int_I v'(r)^2\dr + m|\la|\|v\|_\infty^{q-2}\int_I v(r)^2\dr}
           {\int_I v(r)^2\dr} \\
 &\le \frac{\int_I v'(r)^2\dr + m|\la|\int_I |v(r)|^q\dr}{\int_I v(r)^2\dr} \\
 &\le \frac{M \int_I |v(r)|^p\dr}{\int_I v(r)^2\dr} \\
 &\le M \|v\|_{L^\infty(I)}^{p-2}.
\end{align*}
From this we infer
\begin{align*}
\|v\|_{L^\infty(I)}
 = \Big( \frac{m}{M}\Big)^{\frac1{p-q}}|\la|^{\frac1{p-q}}\cdot \al\qquad \text{where }
  \al^{p-2}-\al^{q-2}
   \ge |I|^{-2}|\la|^{\frac{2-p}{p-q}}\cdot\pi^2 m^{\frac{2-p}{p-q}}M^{\frac{q-2}{p-q}}
\end{align*}
and thus the lower estimate follows from
\begin{align*}
\|v\|_{L^\infty(I)}
 &\ge \Big( \frac{m}{M}\Big)^{\frac1{p-q}}|\la|^{\frac1{p-q}} \cdot
      \max\Big\{1,\Big( |I|^{-2}|\la|^{\frac{2-p}{p-q}} \cdot
            \pi^2m^{\frac{2-p}{p-q}}M^{\frac{q-2}{p-q}} \Big)^{\frac1{p-2}}\Big\} \\
 &= \Big( \frac{m}{M}\Big)^{\frac1{p-q}}|\la|^{\frac1{p-q}} \cdot
      \max\Big\{ 1, |I|^{-\frac2{p-2}}|\la|^{\frac1{p-q}} \cdot
      \pi^{\frac2{p-2}} m^{-\frac1{p-q}}M^{\frac{q-2}{(p-2)(p-q)}} \Big\} \\
 &\ge \Big( \frac{m}{M}\Big)^{\frac1{p-q}}|\la|^{\frac1{p-q}} \cdot
      \frac{\pi^{\frac2{p-2}}m^{-\frac1{p-q}}M^{\frac{q-2}{(p-2)(p-q)}}} {1+\pi^{\frac2{p-2}} m^{-\frac1{p-q}}M^{\frac{q-2}{(p-2)(p-q)}}} \cdot
      (1 + |I|^{-\frac2{p-2}}|\la|^{-\frac1{p-q}} ) \\
 &= \frac{\pi^{\frac2{p-2}} M^{-\frac1{p-2}}}{1+\pi^{\frac2{p-2}}
      m^{-\frac1{p-q}}M^{\frac{q-2}{(p-2)(p-q)}}} \cdot (|\la|^{\frac1{p-q}} +
      |I|^{-\frac2{p-2}})
\end{align*}
where we have used the inequality $\max\{1,xy\}\geq \frac{y}{1+y}(1+x)$ for all $x,y>0$. This proves the last assertion of the Lemma. Using the fact that a solution $(\la,v)\in\Si_j^\pm$ has at least one nodal interval of length $|I|\le \frac{1}{j+1}$ we obtain the lower estimate for $\|v\|_\infty$ from 
$\|v\|_\infty\ge \|v\|_{L^\infty(I)}$. The upper bound for $\|v'\|_\infty$ follows from
\begin{align*}
\|v'\|_\infty
 \le \int_0^1 |v''(r)|\dr
 \le \int_0^1 M\left(|\la||v(r)|^{q-1}+|v(r)|^{p-1}\right)\dr
 \le M g_{|\la|}(\|v\|_\infty).
\end{align*}
Recall that $g_{|\lambda|}$ was defined in \eqref{def:g_la}.

The upper bound for $\|v\|_\infty$ is proved by a blow-up argument. Let us assume for contradiction that there exists a sequence $(\la_n,v_n)$ in $\Si_j^\pm$ with $\la_n\le0$ and ${\|v_n\|_\infty(1+|\la_n|^{\frac1{p-q}})^{-1}\to \infty}$ as $n\to\infty$. We set
$$
  \tilde v_n(r):= t_n^{-1} v_n(r_n+t_n^{-\frac{p-2}2}r)
$$
where $t_n := \|v_n\|_\infty$ and $r_n\in [0,1]$ denotes a maximizer of $|v_n|$. Then $(t_n)_n$ is a sequence
tending to $+\infty$, and we have $|\tilde v_n(0)|=1$ as well as $|\tilde v_n(r)|\le1$, the latter being
defined whenever $0\le r_n+t_n^{-(p-2)/2}r\le1$. Furthermore,
\begin{align*}
-\tilde v_n''(r)
 = t_n^{1-p} h_{\la_n}(r_n+t_n^{-\frac{p-2}{2}} r,t_n\tilde v_n(r),
    |v_n'(r_n+t_n^{-\frac{p-2}{2}}r)|).
\end{align*}

From assumption (B3) we infer that the sequence of functions on the right hand side is bounded. The Arzel\`{a}-Ascoli Theorem provides a subsequence of $(\tilde v_n)_n$ which converges locally uniformly along with its first derivatives to a function 
$\tilde v\in \cC^1(J)$ which is defined on some unbounded interval $J$ containing $0$ and which changes sign at most $j$ times on $J$. Since the sequence 
$(\|\tilde v_n''\|_\infty)_n$ is bounded we even have $\tilde v\in \cC^{1,1}(J)$, i.e.\ $\tilde v'$ is Lipschitz continuous so that $\tilde v''$ exists almost everywhere in $J$. Moreover $\tilde v$ satisfies $|\tilde v(r)|\le 1$ for all $r\in J$,
$\tilde v(0)=1$ as well as
$$
  m |\tilde v|^p \leq -\tilde v'' \tilde v 
  \quad\text{a.e. on }J.
$$
Here we used the lower estimate from assumption (B3) and $t_n\to\infty$. In particular $\tilde v$ is strictly concave on nodal intervals where $\tilde v>0$ and it is strictly convex on nodal intervals where $\tilde v<0$ holds. Now let us show that each nodal interval is bounded. Indeed, if $I$ is such an interval and 
$|\tilde v(\xi)|=\|\tilde v\|_{L^\infty(I)}>0$ then a comparison between $\tilde v$ and
the unique solution of the initial value problem $-\zeta''=m|\zeta|^{p-1}\zeta,\zeta'(\xi)=0,\zeta(\xi)=\tilde v(\xi)$ 
provides a finite upper bound for the length of $I$. As a consequence the union of the nodal intervals is bounded so that $\tilde v$ has to vanish identically on some maximal unbounded interval $J'\subset J,J'\neq J$, the boundary of which has a common point $\eta\in\partial J'$ with some nodal interval. This, however, implies
$\tilde v(\eta)=\tilde v'(\eta)=0$ due to $\tilde v|_J\equiv 0$ as well as 
$\tilde v'(\eta)\neq 0$ due to the strict concavity or convexity on the neighbouring nodal interval. Hence, the assumption was false and the result follows.
\end{proof}

\begin{bem} \label{Bem apriori bounds vs deadcores}
 From the estimates in Lemma~\ref{Lem Apriori estimates I} we deduce that nodal intervals of solutions of \eqref{Gl eq ODE} cannot degenerate within the parameter range $\lambda\in (-\infty,0]$. More precisely we observe that shrinking a nodal interval (i.e. $|I|\to 0$) of $j$-nodal solutions 
 can only occur for ${\lambda\to -\infty}$. In addition, the second estimate in the Lemma implies that there
 is no sequence of solutions $(\lambda_k,v_k)$ with nodal intervals $I_k$ such that $\|v_k\|_{L^\infty(I_k)}$ tends to zero as $k\to\infty$. This is quite remarkable given the fact that dead-core solutions are expected to exist for sufficiently negative $\lambda$. We will use these observations in the proof of Theorem~\ref{Thm 1D case} part~(ii).
\end{bem}

Before we can prove the a-priori estimates for nonnegative $\la$ we provide a technical result which gives some elementary information about the shape of any nontrivial solution of \eqref{Gl eq ODE}. As usual a point $x_0\in [0,1]$ will be called a node of $v\in Y$ if $v(x_0)=0$ and $v'(x_0)\ne 0$.
 The proof of the following Proposition is based on ideas taken from
 \cite[p.~60-61]{WaRe_Radial_solutions_of}.

\begin{prop} \label{Prop 2}
  Let (B1), (B2) hold and let $(\la,v)\in\Si$ with $\la\ge 0$. Then the function $v$ has a finite number of zeros and each zero is a node. In particular the length of all nodal intervals sum up to 1. Moreover, each nodal interval $I$ of $v$ contains a uniquely determined point $\xi\in\mathring{I}$ having the property
  \begin{equation*}
  |v(\xi)| = \|v\|_{L^\infty(I)},\quad
  v'(\xi)=0,\quad
  v'\ne 0\;\;\text{on } I\sm\{\xi\}.
  \end{equation*}
  Moreover, $\max_{I} |v'|$ is attained on $\pa I$.
\end{prop}

\begin{proof}
The second claim follows from the observation that assumption (B2) implies
$$
-v''(r)v(r)=h_\la(r,v(r),|v'(r)|)v(r)>0\quad\text{whenever }\la\ge 0,v(r)\ne 0.
$$
Hence, $|v|$ is strictly concave on every nodal interval which gives the result.

Now it remains to prove that each zero of a given solution $v$ is a node. To this end it suffices to prove $v\equiv0$ if $v$ satisfies $v(r_0)=v'(r_0)=0$ for some $r_0\in[0,1]$. We define $\eta(r):= \int_0^{v(r)}h_\la(r,z,|v'(r)|)\dz$  for $r\in [0,1]$. Multiplying the differential equation with $2v'$ and integrating from $r_0$ to $r$ gives
\begin{align*}
-v'(r)^2
 &= 2\int_{r_0}^r h_\la(t,v(t),|v'(t)|)v'(t)\dt  \\
 &= 2\int_{r_0}^r \eta'(t)\dt - 2\int_{r_0}^r \int_0^{v(t)}
    \pa_rh_\la(t,z,|v'(t)|)+\pa_\xi h_\la(t,z,|v'(t)|)\sign(v'(t))v''(t) \dz\dt \\
 &= 2\eta(r) - 2\int_{r_0}^r \int_0^{v(t)}
      \pa_rh_\la(t,z,|v'(t)|)+\pa_\xi h_\la(t,z,|v'(t)|)\sign(v'(t))v''(t) \dz\dt.
\end{align*}
Notice that differentiation under the integral is justified since (B1) implies that the functions $\pa_rh_\la$ and $\pa_\xi h_\la$ are bounded on compact sets so that the dominated convergence theorem may be applied to the sequence of difference quotients. From (B2) and $\la\ge 0$ we infer that $\eta(r)$ is positive whenever $v(r)\ne 0$. Now choose a sequence $(r_n)$ in $[0,1]$ with $|r_n-r_0|\le \frac{1}{n}$ and
$\eta(r_n)=\max\{\eta(r) : r\in[0,1],|r-r_0|\le \frac{1}{n}\}$. Then we have $\eta(r_n)\to 0$ and assumption (B1) implies for $L:= \max \{|v''(t)|:t\in[0,1]\}$ and some positive number $K$
\begin{align*}
0 &\ge -\frac{1}{2}v'(r_n)^2 \\
  &= \eta(r_n) - \int_{r_0}^{r_n} \int_0^{v(t)}
     \pa_rh_\la(t,z,|v'(t)|) + \pa_\xi h_\la(t,z,|v'(t)|)\sign(v'(t)) v''(t) \dz \dt \\
  &\ge \eta(r_n) - \int_{\min\{r_0,r_n\}}^{\max\{r_0,r_n\}}
        \int_{\min\{0,v(t)\}}^{\max\{0,v(t)\}}
        |\pa_rh_\la(t,z,|v'(t)|)| + L|\pa_\xi h_\la(t,z,|v'(t)|)| \dt \\
  &\ge \eta(r_n)
        - K(1+L) \int_{\min\{r_0,r_n\}}^{\max\{r_0,r_n\}} \int_{\min\{0,v(t)\}}^{\max\{0,v(t)\}}|h_\la(t,z,|v'(t)|)|\dz \dt \\
  &= \eta(r_n)
        - K(1+L) \int_{\min\{r_0,r_n\}}^{\max\{r_0,r_n\}}
        \int_0^{v(t)} h_\la(t,z,|v'(t)|)\dz \dt \\
  &= \eta(r_n)  - K(1+L) \int_{\min\{r_0,r_n\}}^{\max\{r_0,r_n\}} \eta(t) \dt \\
  &\ge \eta(r_n)\cdot (1- K(1+L)|r_n-r_0|).
\end{align*}
This implies $\eta(r_n)=0$ and thus $\eta\equiv0$ on $[0,1]\cap[r_0-\frac{1}{n},r_0+\frac{1}{n}]$ for sufficiently large $n$. Hence, $v$ is trivial in a neighbourhood of $r_0$ and this implies $v\equiv 0$ on $[0,1]$.

Now it remains to prove that each zero of a given solution $v$ is a node. To this end it suffices to prove $v\equiv0$ if $v$ satisfies $v(r_0)=v'(r_0)=0$ for some $r_0\in[0,1]$. We define $\eta(r):= \int_0^{v(r)} h_\la(r,z,|v'(r)|)\dz$  for $r\in [0,1]$. Multiplying the differential equation with $2v'$ and integrating from $r_0$ to $r$ gives
\begin{align*}
-v'(r)^2
 &= 2\int_{r_0}^r h_\la(t,v(t),|v'(t)|)v'(t)\dt  \\
 &= 2\int_{r_0}^r \eta'(t)\dt - 2\int_{r_0}^r \int_0^{v(t)}
    \pa_rh_\la(t,z,|v'(t)|)+\pa_\xi h_\la(t,z,|v'(t)|)\sign(v'(t))v''(t) \dz\dt \\
 &= 2\eta(r) - 2\int_{r_0}^r \int_0^{v(t)}
      \pa_rh_\la(t,z,|v'(t)|)+\pa_\xi h_\la(t,z,|v'(t)|)\sign(v'(t))v''(t) \dz\dt.
\end{align*}
Notice that differentiation under the integral is justified since (B1) implies that the functions $\pa_rh_\la$ and $\pa_\xi h_\la$ are bounded on compact sets so that the dominated convergence theorem may be applied to the sequence of difference quotients. From (B2) and $\la\ge 0$ we infer that $\eta(r)$ is positive whenever $v(r)\ne 0$. Now choose a sequence $(r_n)$ in $[0,1]$ with $|r_n-r_0|\le \frac{1}{n}$ and
$\eta(r_n)=\max\{\eta(r) : r\in[0,1],|r-r_0|\le \frac{1}{n}\}$. Then we have $\eta(r_n)\to 0$ and assumption (B1) implies for $L:= \max \{|v''(t)|:t\in[0,1]\}$ and some positive number $K$
\begin{align*}
0 &\ge -\frac{1}{2}v'(r_n)^2 \\
  &= \eta(r_n) - \int_{r_0}^{r_n} \int_0^{v(t)}
     \pa_rh_\la(t,z,|v'(t)|) + \pa_\xi h_\la(t,z,|v'(t)|)\sign(v'(t)) v''(t) \dz \dt \\
  &\ge \eta(r_n) - \int_{\min\{r_0,r_n\}}^{\max\{r_0,r_n\}}
        \int_{\min\{0,v(t)\}}^{\max\{0,v(t)\}}
        |\pa_rh_\la(t,z,|v'(t)|)| + L|\pa_\xi h_\la(t,z,|v'(t)|)| \dt \\
  &\ge \eta(r_n)
        - K(1+L) \int_{\min\{r_0,r_n\}}^{\max\{r_0,r_n\}} \int_{\min\{0,v(t)\}}^{\max\{0,v(t)\}}|h_\la(t,z,|v'(t)|)|\dz \dt \\
  &= \eta(r_n)
        - K(1+L) \int_{\min\{r_0,r_n\}}^{\max\{r_0,r_n\}}
        \int_0^{v(t)} h_\la(t,z,|v'(t)|)\dz \dt \\
  &= \eta(r_n)  - K(1+L) \int_{\min\{r_0,r_n\}}^{\max\{r_0,r_n\}} \eta(t) \dt \\
  &\ge \eta(r_n)\cdot (1- K(1+L)|r_n-r_0|).
\end{align*}
This implies $\eta(r_n)=0$ and thus $\eta\equiv0$ on $[0,1]\cap[r_0-\frac{1}{n},r_0+\frac{1}{n}]$ for sufficiently large $n$. Hence, $v$ is
trivial in a neighbourhood of $r_0$ and this implies $v\equiv 0$ on $[0,1]$.
\end{proof}

Now let us prove the a-priori estimates for positive $\la$. The idea of our approach comes from the analysis of the one-dimensional Ambrosetti-Brezis-Cerami problem
\begin{equation} \label{Gl 1D ABC-problem}
\left\{
\begin{aligned}
&-u'' = \la |u|^{q-2}u + |u|^{p-2}u \quad\text{in }(0,1),\\
&u(0)=u(1)=0,
\end{aligned}
\right.
\end{equation}
where the existence and the precise shape of the solution continua $\cD_j^\pm$ enjoying the properties (i)-(v) from Theorem~\ref{Thm 2} can be proved using the so-called energy method or time-map technique. We refer the interested reader to Appendix~A for the proof of the corresponding result, cf.\ Theorem~\ref{Thm 1D case}. The primitive of the map $g_\la$ from \eqref{Gl def mMKc} is denoted by
$$
  G_\la(z) = \frac{\la}{q}|z|^q + \frac{1}{p}|z|^p \qquad\text{for }z\in\R.
$$
For $\la>0$ the "time map" $T_\la:\R_{>0}\to\R_{>0}$ associates to $\al>0$ the first time $t>0$ such that a solution $u$ of
\begin{equation}\label{eq:model}
-u''=g_\la(u),\quad u(0)=0,\quad \|u\|_\infty =\al
\end{equation}
satisfies $|u(t)|=\|u\|_\infty =\al$. Observe that such a solution is uniquely defined up to sign. The time map is given explicitly by
\begin{equation} \label{Gl def timemap}
T_\la(\al) = \int_0^\al \frac{1}{\sqrt{2(G_\la(\al)-G_\la(z)})}\,dz.
\end{equation}
Since a solution $u$ of \eqref{eq:model} with $u(T)=\|u\|_\infty$ satisfies $u(T+t)=u(T-t)$, it follows that a solution $\al$ of $T_\la(\al)=\frac{1}{2j+2}$ yields a $j$-nodal solution of \eqref{Gl 1D ABC-problem} with
$|u(\frac{1}{2j+2})|=\|u\|_\infty =\al$. This fact will be proved in Proposition~\ref{Prop App A 3}. In the analysis of the nonautonomous boundary value problem \eqref{Gl eq ODE}, however, such an exact solution theory is out of reach since the a-priori information about the localization of maximizers of an arbitrary solution of \eqref{Gl eq ODE} with $j+1$ nodal intervals is not available. Nevertheless we find a weaker result stating that every such solution of \eqref{Gl eq ODE} satisfies
\begin{align} \label{Gl Definition untj ovtj}
\frac{\sqrt{m}\cdot s_j}{2} \le T_\la(\|v\|_\infty)
 \le \frac{\sqrt{M}\cdot t_j}{2}
 \qquad\text{where }s_j = \frac{1}{a^j(j+1)}, \quad t_j := 1-js_j.
\end{align}
Here, the number $a>1$ is given by
\begin{align} \label{Gl a}
a := \max\Big\{ \Big(\frac{M}{m}\Big)^{1/q}(C_1C_2)^{(2-q)/2q},
      \Big(\frac{M}{m}\Big)^{(p-1)/p}(C_1C_2)^{(p-2)/2p}\Big\}
\end{align}
and the positive numbers $C_1,C_2\ge 1$ will be provided in 
Proposition~\ref{Prop Vorbereitung AprioriII 2}. In the proof of the a-priori estimates for $\la\ge 0$ we will need several properties of the time map $T_\la$ which we summarize in the following two Propositions~\ref{Prop Vorbereitung AprioriII 1} and~\ref{Prop Vorbereitung AprioriII 2} the proofs of which we defer to Appendix~B. These prelimary results tell us that for positive $\lambda$ the time map $T_\lambda$ may be qualitatively depicted as in Figure~\ref{pic Tlambda lambdageq0}.

\medskip

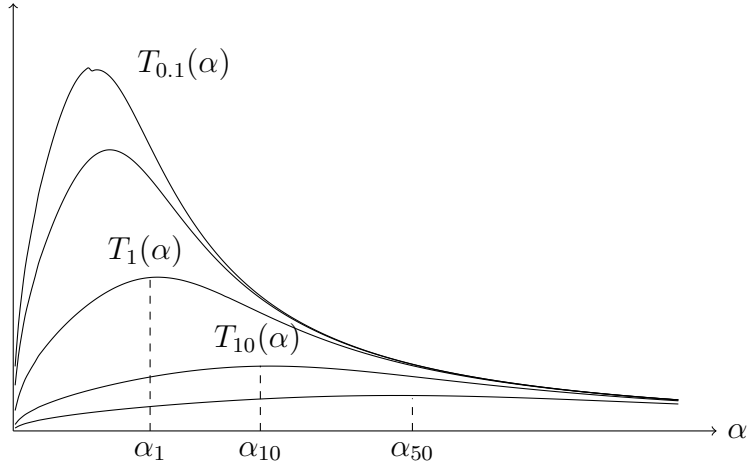
\begin{figure}[!htp]
 \centering
    \begin{tikzpicture}[domain=0.01:3.5, xscale=2.5, yscale=2.7, samples=200]
      \draw[->] (0,0) -- (3.7,0) node[right]{$\al$};
      \draw[->] (0,0) -- (0,2.1); 
      \draw[smooth] plot (\x,{exp(ln(\x)*1/2)/exp(ln(50+(\x)^4)*1/2)}); 
      \draw[smooth] plot (\x,{exp(ln(\x)*1/2)/exp(ln(10+(\x)^4)*1/2)}); 
      \draw[smooth] plot (\x,{exp(ln(\x)*1/2)/exp(ln(1+(\x)^4)*1/2)});  
      \draw[smooth] plot (\x,{exp(ln(\x)*1/2)/exp(ln(0.2+(\x)^4)*1/2)}); 
      \draw[smooth] plot (\x,{exp(ln(\x)*1/2)/exp(ln(0.1+(\x)^4)*1/2)}); 
      \draw[dashed] (0.72,0) node[below] {$\al_1$} -- (0.72,0.75)  node[above]{$T_1(\al)$\;};
      \draw[dashed] (1.3,0) node[below] {$\al_{10}$} -- (1.3,0.32) node[above]{$T_{10}(\al)$\,};
      \draw[dashed] (2.1,0) node[below] {$\al_{50}$} -- (2.1,0.16);
      \node at (0.9,1.8) {$T_{0.1}(\al)$};
      \label{pic a}
    \end{tikzpicture} 
 \caption{Qualitative plots of $T_\la$ for $q=1.5,p=4$ and $\la\in\{0.1,0.2,1,10,50\}$.}
  \label{pic Tlambda lambdageq0} 
\end{figure} 
 
\medskip

\begin{prop}\label{Prop Vorbereitung AprioriII 1}
 For all $\la>0$ there is a uniquely determined $\al_\la>0$ such that $T_\la$ is strictly increasing on $(0,\al_\la]$ and strictly decreasing on $[\al_\la,\infty)$. Moreover, we have
 $$
 \lim_{\al\to 0^+} T_\la(\al) = \lim_{\al\to \infty} T_\la(\al) = 0
 $$
 and there is a positive number $C_3\ge 1$ such that
 $T_\la(\al_\la)\le C_3\la^{(2-p)/2(p-q)}$ for all $\la>0$.
\end{prop}

\begin{bem}\label{Bem T0}
Explicit calculations give for $\al>0$:
$$
T_0(\al)
 = \sqrt{\frac{p}{2}} \int_0^1 \frac{1}{\sqrt{1-s^p}}\ds \cdot \al^{(2-p)/2}.
$$
In particular, the time map $T_0$ does not have the properties described in Proposition~\ref{Prop Vorbereitung AprioriII 1}. This corresponds to the fact that the bifurcation result from Theorem~\ref{Thm 2}~(iv)(a) is not true for $\la=0$ while it is true for all $\la>0$.
\end{bem}


\begin{prop}\label{Prop Vorbereitung AprioriII 2}
There are positive numbers $C_1,C_2\ge 1$ only depending on $p,q$ such that for all $\la\ge 0$ the following estimates hold:
\begin{align*}
\frac{G_\la(\be_1)}{G_\la(\be_2)}
 &\le C_1 \Big(\frac{T_\la(\be_1)}{T_\la(\be_2)}\Big)^{2q/(2-q)}
  && \text{for }0<\be_1,\be_2\le \al_\la, \\
      \frac{G_\la(\be_1)}{G_\la(\be_2)}
 &\le C_2 \Big(\frac{T_\la(\be_2)}{T_\la(\be_1)}\Big)^{2p/(p-2)}
  && \text{for }\be_1,\be_2\ge \al_\la, \\
\frac{G_\la(\be_1)}{G_\la(\be_2)}
 &\le C_1C_2 \min\Big\{
       \Big(\frac{T_\la(\be_1)}{T_\la(\al_\la)}\Big)^{2q/(2-q)},
       \Big(\frac{T_\la(\be_2)}{T_\la(\al_\la)}\Big)^{2p/(p-2)}\Big\}
  &&\text{for }\be_1\le \al_\la\le  \be_2.
\end{align*}
\end{prop}

Next let us use these technical propositions in order to prove a priori estimates for $j$-nodal solutions of~\eqref{Gl eq ODE}.

\begin{lem} \label{Lem Apriori estimates II}
Assume that (B1), (B2) hold, and fix $j\in\N_0$. Then the following estimates hold for all $(\la,v)\in\Si_j^\pm$ with $\la\ge 0$:
\begin{align*}
\la \le \Big( \frac{2C_3(j+1)}{m}\Big)^{2(p-q)/(p-2)},\quad
\frac{\sqrt{m}\cdot s_j}{2} \le T_\la(\|v\|_\infty)
 \le \frac{\sqrt{M}\cdot t_j}{2}, \quad
\|v'\|_\infty \le \sqrt{2MG_\la(\|v\|_\infty)}
\end{align*}
where $s_j$, $t_j$ are given by \eqref{Gl Definition untj ovtj}.
\end{lem}

\begin{proof}
Let $(\la,v)\in\Si_j$ be a solution of \eqref{Gl eq ODE} with closed nodal intervals
$I_0,\ldots,I_{j}$ of length ${l_0,\ldots,l_{j}>0}$ satisfying $l_0+\ldots+l_{j}=1$; see Proposition~\ref{Prop 2}. We set $\al:= \|v\|_{L^\infty([0,1])}$ and $\be_i:=\|v\|_{L^\infty(I_i)}$ for $i=0,\ldots,j$. Proposition \ref{Prop 2} implies that every nodal interval $I_i=:[\eta_i,\eta_i+l_i]$ contains a uniquely determined point $\xi_i\in \mathring{I_i}$ having the property
\begin{equation} \label{Gl Apriori 1}
|v(\xi_i)|=\be_i,\qquad
v'(\xi_i)=0,\qquad
v'\ne 0\quad\text{on } [\eta_i,\eta_i+l_i]\sm\{\xi_i\},
\end{equation}
and $\max_{I_i} |v'|$ is attained on $\pa I_i$.

{\it Step 1: Estimate for $\be_i$ in terms of $l_i$.}\; Our first aim is to prove for $i=0,\ldots,j$ the estimates
\begin{equation}\label{Gl Apriori 2}
\frac{\sqrt{m}l_i}{2} \le T_\la(\be_i) \le \frac{\sqrt{M}l_i}{2}
\quad\text{and}\quad
\sqrt{2mG_\la(\be_i)} \le |v'|_{\pa I_i}| \le  \sqrt{2M G_\la(\be_i)}.
\end{equation}
We only prove the assertion in case $v|_{I_i}>0$ since the reasoning for the case $v|_{I_i}<0$ is similar. From \eqref{Gl Apriori 1} and (B2) we get
\begin{align*}
-v''v' &= h_\la(r,v,|v'|)v' \ge m g_\la(v)v' \quad\text{on }[\eta_i,\xi_i], \\
\noalign{and}
-v''v' &= h_\la(r,v,|v'|)v' \le m g_\la(v)v' \quad\text{on }[\xi_i,\eta_i+l_i].
\end{align*}
Integrating these inequalities from $x$ to $\xi_i$ for all $x\in I_i$ gives
\begin{align} \label{Gl Apriori 6}
(v')^2 &\ge 2m  (G_\la(\be_i)-G_\la(v)) \quad\text{on }I_i.
\end{align}
From this inequality and \eqref{Gl Apriori 1} we infer
\begin{align}\label{Gl Apriori 3}
2T_\la(\be_i)
 &= \int_0^{\be_i} \frac{1}{\sqrt{2(G_\la(\be_i)-G_\la(z))}} \dz
     + \int_0^{\be_i} \frac{1}{\sqrt{2(G_\la(\be_i)-G_\la(z))}} \dz \notag\\
 &= \int_{v(\eta_i)}^{v(\xi_i)} \frac{1}{\sqrt{2(G_\la(\be_i)-G_\la(z))}} \dz
     + \int_{v(\eta_i+l_i)}^{v(\xi_i)} \frac{1}{\sqrt{2(G_\la(\be_i)-G_\la(z))}} \dz \notag\\
 &= \int_{\eta_i}^{\xi_i} \frac{v'(t)}{\sqrt{2(G_\la(\be_i)-G_\la(v(t)))}} \dt
     + \int_{\xi_i}^{\eta_i+l_i} \frac{-v'(t)}{\sqrt{2(G_\la(\be_i)-G_\la(v(t)))}} \dt \notag\\
 &\ge \int_{\eta_i}^{\xi_i} \sqrt{m} \dt + \int_{\xi_i}^{\eta_i+l_i} \sqrt{m} \dt
      \notag\\
 &= \sqrt{m} l_i
\end{align}
and similarly one proves
\begin{align}\label{Gl Apriori 7}
(v')^2 &\le 2M (G_\la(\be_i)-G_\la(v)) \quad\text{on }I_i,\qquad
2T_\la(\be_i) \le \sqrt{M} l_i.
\end{align}
The inequalities \eqref{Gl Apriori 6}--\eqref{Gl Apriori 7} and
 $G_\lambda(v)|_{\pa I_i} = v|_{\pa I_i}=0$ imply \eqref{Gl Apriori 2}.

{\it Step 2: Estimate for $\|v'\|_\infty$.}\;  Proposition \ref{Prop 2} implies
$\|v'\|_{L^\infty(I_i)}=\max_{\pa I_i} |v'|$ for all $i\in\{0,\ldots,j\}$. From \eqref{Gl Apriori 2} and the monotonicity of $g_\la,G_\la$ we deduce
\begin{align*}
\|v'\|_{L^\infty([0,1])}
 &= \max_{i=0,\ldots,j} \|v'\|_{L^\infty(I_i)}
  \le \max_{i=0,\ldots,j} \sqrt{2MG_\la(\be_i)}
  = \sqrt{2MG_\la(\al)}.
\end{align*}

{\it Step 3: Estimate for $\la$.}\;
From $l_0+\ldots+l_{j}=1$, \eqref{Gl Apriori 2} and the inequality
$T_\la(\al_\la)\le C_3\la^{(2-p)/2(p-q)}$ from Proposition 
\ref{Prop Vorbereitung AprioriII 1} we obtain
\begin{equation}\label{GL_est-T_lambda}
\frac{\sqrt{m}}{2(j+1)}
 \le \frac{\sqrt{m}}{2} \max_{k=0,\ldots,j} l_k
 \le T_\la(\al_\la)
 \le C_3\la^{(2-p)/2(p-q)}
\end{equation}
and hence the estimate for $\la$.

{\it Step 4: Estimate for $l_0,\ldots,l_{j}$ and $\|v\|_\infty$.}\;
From \eqref{Gl Apriori 2} we obtain the estimate for $\|v\|_\infty$ once we have shown the inequality
\begin{equation}\label{Gl Apriori 8}
s_j \le l_i \le t_j \qquad \text{for all } i\in\{0,\ldots,j\}.
\end{equation}
Since this estimate is trivial in case $j=0$ we only consider the case $j\ge 1$. Due to $l_0+\ldots+l_{j}=1$ and $t_j=1-js_j$, see \eqref{Gl Definition untj ovtj}, it suffices to prove the lower estimate. Since $\pa I_i\cap \pa I_{i+1}$ is non-empty for all $i\in\{0,\ldots,j-1\}$ we infer from \eqref{Gl Apriori 2} that the intervals    $\left[\sqrt{2mG_\la(\be_i)},\sqrt{2MG_\la(\be_i)}\right]$ and    $\left[\sqrt{2mG_\la(\be_{i+1})},\sqrt{2MG_\la(\be_{i+1})}\right]$ overlap. In particular this entails
\begin{align} \label{Gl Apriori 9}
\frac{m}{M}\le \frac{G_\la(\be_i)}{G_\la(\be_{i+1})}\quad\text{or}\quad
\frac{m}{M}\le \frac{G_\la(\be_{i+1})}{G_\la(\be_i)}.
\end{align}
In view of the case distinction from Proposition~\ref{Prop Vorbereitung AprioriII 2} we define the covering $\{J_1,J_2\}$ of the index set $\{0,\ldots,j-1\}$ as follows:
\begin{align*}
J_1 &:= \{i\in\{0,\ldots,j-1\}: \be_i\le \al_\la\le \be_{i+1}
        \;\text{ or }\;\be_{i+1}\le \al_\la\le \be_i \}. \\
J_2 &:= \{ i\in\{0,\ldots,j-1\}: \be_i,\be_{i+1}\le \al_\la
        \;\text{ or }\; \al_\la\le \be_i,\be_{i+1} \}.
\end{align*}

{\it 1st case: $J_2=\emptyset$.}\; From Proposition~\ref{Prop Vorbereitung AprioriII 2} we get for all $i\in\{0,\ldots,j-1\}$:
\begin{align*}
\frac{m}{M}
 \le C_1C_2 \Big( \frac{T_\la(\be_i)}{T_\la(\al_\la)}\Big)^{2q/(2-q)}
      \quad\text{if } \be_i\le \al_\la, \\
\frac{m}{M}
 \le C_1C_2 \Big( \frac{T_\la(\be_i)}{T_\la(\al_\la)}\Big)^{2p/(p-2)}
      \quad\text{if } \be_i\ge \al_\la.
\end{align*}
From \eqref{Gl Apriori 2} and \eqref{GL_est-T_lambda} we deduce:
\begin{align*}
\frac{\sqrt{M}l_i}{2}
 &\ge T_\la(\be_i)
  \ge  T_\la(\al_\la) \cdot \Big( \frac{m}{MC_1C_2}\Big)^{(2-q)/2q}
  \ge \frac{\sqrt{m}}{2(j+1)}\Big( \frac{m}{MC_1C_2}\Big)^{(2-q)/2q}
      \quad\text{if }  \be_i\le \al_\la,\\
\frac{\sqrt{M}l_i}{2}
 &\ge T_\la(\be_i)
  \ge  T_\la(\al_\la) \cdot \Big( \frac{m}{MC_1C_2}\Big)^{(p-2)/2p}
  \ge \frac{\sqrt{m}}{2(j+1)}\Big( \frac{m}{MC_1C_2}\Big)^{(p-2)/2p}
      \quad\text{if } \be_i\ge \al_\la,
\end{align*}
which, by definition of $a$ and $s_j$ in \eqref{Gl Definition untj ovtj}, in particular implies
$$
l_i
 \ge \frac{1}{j+1}\cdot \min\Big\{ \Big(\frac{m}{M}\Big)^{1/q}
     (C_1C_2)^{(q-2)/2q},
     \Big(\frac{m}{M}\Big)^{(p-1)/p} (C_1C_2)^{(2-p)/2p} \Big\}
 = \frac{1}{a(j+1)}
 \ge s_j
$$
so that \eqref{Gl Apriori 8} is proved in this special case.

{\it 2nd case: $J_1=\emptyset$.}\;
Using the estimates from \eqref{Gl Apriori 9}, 
Proposition~\ref{Prop Vorbereitung AprioriII 2} and \eqref{Gl Apriori 2} we obtain for all $i\in\{0,\ldots,j-1\}$:
\begin{align*}
\frac{m}{M}
 &\le C_1 \Big( \frac{\sqrt{M}l_i}{\sqrt{m}l_{i+1}}\Big)^{2q/(2-q)}, \quad
\frac{m}{M}
  \le C_1\Big( \frac{\sqrt{M} l_{i+1}}{\sqrt{m} l_i}\Big)^{2q/(2-q)}
       \qquad \text{if }\be_i,\be_{i+1}\le \al_\la, \\
\frac{m}{M}
 &\le C_2 \Big( \frac{\sqrt{M} l_{i+1}}{\sqrt{m}l_i}\Big)^{2p/(p-2)}, \quad
\frac{m}{M}
  \le C_2\Big(  \frac{\sqrt{M}l_i}{\sqrt{m}l_{i+1}}\Big)^{2p/(p-2)}
       \qquad \text{if }\be_i,\be_{i+1}\ge \al_\la.
\end{align*}
In both cases the choice for $a$ from \eqref{Gl a} and $C_1,C_2\ge 1$ imply
$a^{-1} l_{i+1} \le l_i \le a l_{i+1}$ for all $i\in\{0,\ldots,j-1\}$ and thus
$$
1 = l_0+\ldots+l_j
 \le \min\{l_0,\ldots,l_{j}\} \cdot (1+a+\ldots+a^j)
 \le  \min\{l_0,\ldots,l_{j}\} \cdot \frac{1}{s_j}
$$
which gives \eqref{Gl Apriori 8}.

{\it 3rd case: $J_1\ne \emptyset,J_2\ne \emptyset.$}\; As in the previous cases we obtain $a^{-1} l_{i+1} \le l_i\le a l_{i+1}$ for all $i\in J_1$ as well as $l_k,l_{k+1}\ge\frac{1}{a(j+1)}\ge s_j$ for all $k\in J_2$. For any given $i\in J_1$ choose $k(i)\in J_2$ such that $|k(i)-i|=\min\{ |k-i|: k\in J_2\}$. Then we have
\begin{align*}
l_i &\ge a^{-1}l_{i+1} \ge \ldots \ge a^{-(k(i)-i)}l_{k(i)}  \qquad
 \ge a^{-(j-1)}\cdot \frac{1}{a(j+1)} = s_j
  &&\text{in case }k(i)>i,\\
l_i &\ge a^{-1}l_{i-1} \ge \ldots \ge a^{-(i-k(i)-1)}l_{k(i)+1}
 \ge a^{-(j-1)}\cdot \frac{1}{a(j+1)} = s_j
  &&\text{in case }k(i)<i,
\end{align*}
which is all we had to show. This finishes the proof.
\end{proof}

%
%

\section{Proof of Theorem \ref{Thm 1} and Theorem \ref{Thm 2}} \label{sec:proof-thm2}
Since we are going to use Leray-Schauder degree theory we introduce the solution operator $S_\lambda:Y\to Y$
associated to the boundary value problem \eqref{Gl eq ODE}. Using the Green's function
$$
  G:[0,1]\times [0,1]\to\R,\quad (x,s)\mapsto
 \begin{cases}
       x(1-s) &\text{if } 0\le x\le s\le 1, \\
       s(1-x) &\text{if } 0\le s\le x\le 1,
 \end{cases}
$$
of the differential operator $-\frac{d^2}{d x^2}$ on the interval $(0,1)$ associated to homogeneous Dirichlet boundary conditions we define
$$
S_\la(v) := v - \int_0^1 G(\cdot,s)h_\la(s,v(s),|v'(s)|)\ds.
$$
This defines a continuous and compact perturbation of the identity. A solution $(\la,v)$ of $S_\la(v)=0$ is a classical solution of \eqref{Gl 1st order eq}. In order to prove that the degree of $S_0$ over a suitable open set (whose closure does not contain the trivial solution; see \eqref{Gl Defn V}) is non-zero we introduce the homotopy ${H:Y\times [0,1]\to Y}$ given by
\begin{equation} \label{Gl Definition homotopy}
H(v,t)
 := v - \int_0^1 G(\cdot,s)\Big(t \cdot m|v(s)|^{p-2}v(s)+(1-t)\cdot
    h_0(s,v(s),|v'(s)|)\Big)\ds
\end{equation}
which relates the original boundary value problem \eqref{Gl eq ODE} for $\la=0$ to the
autonomous boundary value problem
\begin{equation} \label{Gl Hilfs-BVP}
-v'' = m |v|^{p-2}v,\qquad v(0)=v(1)=0.
\end{equation}
Here, the positive number $m$ is given by assumption (B2), (B3). The following result is well-known, we include a proof for completeness.

\begin{prop} \label{Prop vjstar}
 For all $j\in\N_0$ the boundary value problem \eqref{Gl Hilfs-BVP} has a unique solution with precisely $j$ nodes in (0,1) and positive slope at 0. This solution, called $\zeta_j$, satisfies
 $$
   \|\zeta_j\|_\infty
   = \Big((j+1) \sqrt{\frac{2p}{m}}  \int_0^1 (1-s^p)^{-1/2}\ds\Big)^{2/(p-2)}
 $$
 and it is nondegenerate, i.e.\ the boundary value problem
 $$
 -\phi'' = (p-1)m |\zeta_j|^{p-2}\phi,\qquad \phi(0)=\phi(1)=0
 $$
 only has the trivial solution.
\end{prop}

\begin{proof}
After rescaling we may assume $m=1$. Observe that \eqref{Gl Hilfs-BVP} has a unique positive solution $\zeta_0$ which extends to an odd, 2-periodic function on $\R$. Then
$\zeta_j(x)=(j+1)^\frac{2}{p-2}\zeta_0((j+1)x)$ is the unique  solution of 
\eqref{Gl Hilfs-BVP} with $j$ nodes and $\zeta_j'(0)>0$. Notice that the explicit formula for $T_0$ from Remark \ref{Bem T0} and the equation $T_0(\|\zeta_0\|_\infty)=\frac12$
yield the desired formula for $\|\zeta_0\|_\infty$, hence for $\|\zeta_j\|_\infty$.

It remains to prove non-degeneracy of $\zeta_j$. Since for all $\al>0$ the function $v_\al(x):= \al^{2/(p-2)}\zeta_j(\al x)$ satisfies $v_\al(0)=0$ and $-v_\al''=|v_\al|^{p-2}v_\al$  we obtain that the function
$\varphi := \frac{d}{d\al}|_{\al=1} v_\al$ spans the one-dimensional linear space consisting of all solutions of the linear differential equation
$-\phi'' = (p-1)|\zeta_j|^{p-2}\phi$ with $\phi(0)=0$. From $\zeta_j(1)=0$ and $\zeta_j'(1)\ne 0$ we infer
\begin{align*}
\varphi(1) = \frac{d}{d\al}\Big|_{\al=1} v_\al(1)
 = \frac{2}{p-2} \zeta_j(1) +\zeta_j'(1)
 \ne 0.
\end{align*}
Hence the boundary value problem
$$
-\phi'' = (p-1)|\zeta_j|^{p-2}\phi,\qquad \phi(0)=\phi(1)=0,
$$
only has the trivial solution which proves that $\zeta_j$ is non-degenerate.
\end{proof}

\begin{altproof}{Theorem \ref{Thm 2}}
We fix $j\in\N_0$ and prove the assertion for $\cD_j^+$ only; the proof for $\cD_j^-$ proceeds analogously. For $t\in [0,1]$ we define
$$
h^t:[0,1]\times\R\times\R_{\ge 0} \to \R,\quad
h^t(r,z,\xi) = t\cdot m|z|^{p-2}z+(1-t)\cdot h_0(r,z,\xi)
$$
so that solving $H(v,t)=0$ is equivalent to solving the boundary value problem
$$
-v''(r) = h^t(r,v(r),|v'(r)|),\qquad v(0)=v(1)=0.
$$
Since $h^t$ satisfies (B1) as well as the inequality from (B2) for $\la=0$, the a-priori estimates from
Lemma~\ref{Lem Apriori estimates II} for $\la=0$ yield $0\notin H(\pa V_j^+,[0,1])$ for the homotopy $H$ from
\eqref{Gl Definition homotopy} where the bounded open set $V_j^+\subset Y$ is given by
\begin{equation}\label{Gl Defn V}
\begin{aligned}
V_j^+ &:= \Big\{ v\in Y: v\ne 0 \text{ has precisely }j+1 \text{ nodal intervals }
        I_0,\ldots,I_j,\ \sign(v|_{I_k})=(-1)^k \\
  &\hspace{2cm} \text{and }
      \frac{\sqrt{m}\cdot s_j}{4} < T_0(\|v\|_\infty) < \sqrt{M}\cdot t_j,\,
      \|v'\|_\infty  < 2\sqrt{MG_0(\|v\|_\infty)}
      \Big\}.
\end{aligned}
\end{equation}
Moreover, Proposition \ref{Prop vjstar} gives that $\zeta_j$ is the only solution of $H(v,1)=0$ in $V_j^+$ and that this solution is non-degenerate. Hence, the homotopy invariance of the Leray-Schauder degree yields
\begin{equation} \label{Gl degree nonzero}
\deg(S_0,V_j^+,0)
 = \deg(H(\cdot,0),V_j^+,0) = \deg(H(\cdot,1),V_j^+,0)
 = \ind(H(\cdot,1),\zeta_j) \ne 0.
\end{equation}
Now let $\cD_j^+$ denote the maximal connected set in $\Si_j$ containing the nonempty set of all solutions $(0,v)$ of $S_0(v)=0$ with $v\in V_j^+$. We define the subcontinua $\cD_{j,\le0}^{+},\cD_{j,\ge 0}^{+}$ by
$$
\cD_{j,\le 0}^{+}:= \{ (\la,v)\in\cD_j^+ : \la\le 0\},\qquad
\cD_{j,\ge 0}^{+}:= \{ (\la,v)\in\cD_j^+ : \la\ge 0\}.
$$
For a better understanding of the following proofs of the parts (i)-(iii) we include Figure~\ref{Fig
apriori bounds} which depicts the a priori bounds using dashed lines and the boundary of the open sets
$\mathcal{O}$ appearing in (i),(iii) using dotted lines.

\medskip

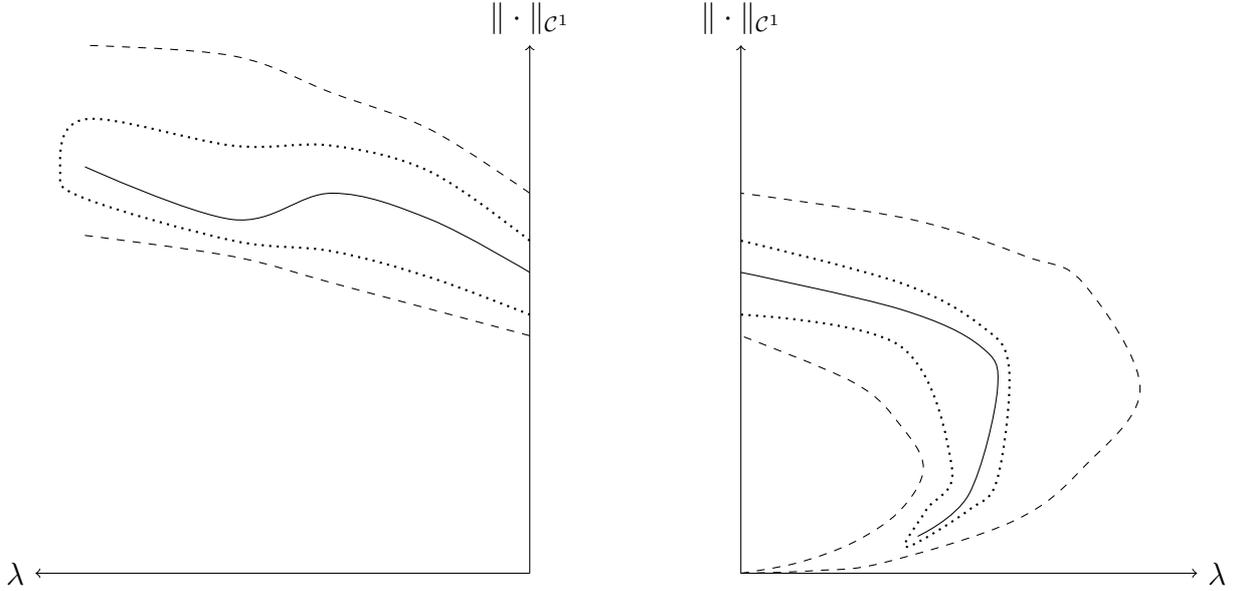
\begin{figure}[!htp]
\centering
 \subfigure
 {
  \begin{tikzpicture}[yscale=0.7, xscale=0.65]
	  \draw[->] (0,0) -- (-10,0) node[left] {$\la$};
	  \draw[->] (0,0) -- (0,10)  node[above] {$\|\cdot\|_{\cC^1}$};
	  \draw [dotted,thick] plot[smooth, tension=0.6] coordinates { (0,6.3)  (-2,7.6) (-4,8.1) (-6,8.1) (-9,8.6)
  	   (-9.5,7.6) (-9,7.1) (-6,6.3) (-4,6.1) (-2,5.6) (0,4.9)};
  	  \draw plot[smooth, tension=0.6] coordinates {(0,5.7) (-2,6.7) (-4,7.2) (-6,6.7) (-9,7.7)};
  	  \draw [dashed] plot[smooth, tension=0.6] coordinates {  (0,7.2)  (-2,8.4) (-4,9.1) (-6,9.8) (-9,10)};
  	  \draw [dashed] plot[smooth, tension=0.6] coordinates {  (0,4.5)  (-2,5) (-4,5.5) (-6,6) (-9,6.4)};
  \end{tikzpicture}
 }
 \hfill\hspace{-2cm}
 \subfigure
 {
  \begin{tikzpicture}[yscale=0.7, xscale=0.75]
	  \draw[->] (0,0) -- (8,0) node[right] {$\la$};
	  \draw[->] (0,0) -- (0,10)  node[above] {$\|\cdot\|_{\cC^1}$};
  	  \draw [dotted,thick] plot[smooth, tension=0.6] coordinates {  (0,6.3) (2.8,5.5) (4.2,4.7) (4.7,3.9)
  	  (4.5,1.8) (4,1.2) (2.9,0.5) (3.25, 1.2)  (3.7,2) (2.8,4.3) (0,4.9)};
  	  \draw plot[smooth, tension=0.6] coordinates {(3.1,0.7) (4,1.5) (4.5,3.5) (4.2,4.3) (2.8,5) (0,5.7)};
  	  \draw [dashed] plot[smooth, tension=0.6] coordinates {  (0,0)  (1,0.2) (2,0.6) (2.8,1.2) (3.2,2)
  	  (2.8,2.8) (2,3.6) (0,4.5)};
  	  \draw [dashed] plot[smooth, tension=0.6] coordinates {  (0,0)  (1.3,0.05) (2.5,0.2) (4.8,1) (6,2) (7,3.5)
  	  (6,5.5) (5,6) (3,6.7) (0,7.2)};
  \end{tikzpicture}
 }
\caption{A priori bounds, the open sets $\mathcal{O}$ and the solution continua}
\label{Fig apriori bounds}
\end{figure}

\medskip

{\it Proof of (i):\;}
Since $\cD_j^+\subset\Si_j$ is connected and contains solutions of \eqref{Gl eq ODE} lying in $\{0\}\times V_j^+$ we deduce that every $(\la,v)\in\cD_j^+$ belongs to $\Si_j^+$. Indeed, one may use Proposition~\ref{Prop 2} and 
Lemma~\ref{Lem Apriori estimates I} to show that the set of nontrivial solutions with $j+1$ nodal intervals $I_0,\ldots,I_j$ and sign $(-1)^k$ on $I_k$ is open and closed in $\Si$; see Remark~\ref{Bem apriori bounds vs deadcores}. Furthermore, the a-priori estimates from Lemma \ref{Lem Apriori estimates I} and 
Lemma~\ref{Lem Apriori estimates II} imply that for all $\eps>0$ there is a positive
number $c_\eps$ such that $\|v\|_\infty\ge c_\eps$ for all $(\la,v)\in \Si_j^+$ with $|\lambda|\geq \eps$. Hence, we have $(\la,0)\notin\ov{\Si}_j^+$ for all $\la\ne 0$ and thus
$$
  \ov{\cD_j^+} \sm \cD_j^+ \subset \{(0,0)\}.
$$

It remains to show $(0,0)\in\ov{\cD_{j,\ge 0}^+}$ for all $j\in\N_0$. If this were not true then the a-priori estimates from Lemma \ref{Lem Apriori estimates II} and Whyburn's Lemma (see for example Lemma 3.5.2 in \cite{Cha_Methods_in_nonlinear}) would provide a bounded relatively open set $\cO$ in $\R_{\ge0}\times Y$ such that $\cD_{j,\ge0}^+\subset\cO$, $(0,0)\notin\cO$, and such that the relative boundary of $\cO$ does not contain any solution of the equation $S_\la(v)=0$; see
Figure~\ref{Fig apriori bounds}. Since the slice $(\cD_j^+)_0$ contains all solutions of $S_0(v)=0$ in $V_j^+$ so does $\cO_0$, and we deduce, using the excision property and the generalized homotopy invariance of the Leray-Schauder degree,
$$
  \deg(S_0,V_j^+,0) = \deg(S_0,\cO_0,0) = \lim_{\la\to +\infty} \deg(S_\la,\cO_\la,0)
   = 0
$$
which contradicts \eqref{Gl degree nonzero}. Hence, we obtain $(0,0)\in\ov{\cD_j^+}$ and claim (i) is proved.

{\it Proof of (ii):\;} This follows immediately from (i) because the sets $\cD_j^\pm$, $\cD_k^\pm$ are pairwise disjoint by definition.

{\it Proof of (iii):\;} We first show that $\cD_{j,\le 0}^{+}$ is unbounded to the left. Otherwise we could use Whyburn's Lemma and the a-priori estimates from Lemma~\ref{Lem Apriori estimates I} to find a bounded relatively open set $\cO$ in $\R_{\le 0}\times Y$ with $\cD_{j,\le 0}^{+}\subset \cO$ such that the relative boundary of $\cO$ does not contain any solution of \eqref{Gl eq ODE}. As above we get
$$
  \deg(S_0,V_j^+,0) = \deg(S_0,\cO_0,0) = \lim_{\la\to -\infty} \deg(S_\la,\cO_\la,0) = 0
$$
which contradicts \eqref{Gl degree nonzero}. Hence, we obtain
\begin{equation} \label{Gl pr negative part}
\pr(\cD_{j,\le 0}^{+})= (-\infty,0].
\end{equation}
Now let us show that the positive numbers $\La_j^+:=\max\pr(\cD_{j,\ge0}^{+})$ tend to infinity as $j\to\infty$. The continuum $\cD_{j,\ge0}^+$ contains a solution $(0,v_j)$ with $v_j\in V_j^+$ and \eqref{Gl Defn V} implies $T_0(\|v_j\|_\infty)<\sqrt{M}\cdot t_j$. The formula for $T_0$ from Remark \ref{Bem T0} and $ t_j\le 1$ give
\begin{equation} \label{Gl Proof of Them part (iii)}
\|v_j\|_\infty
 \ge \Big(\frac{1}{ t_j}\sqrt{\frac{p}{2M}} \int_0^1 \frac{1}{\sqrt{1-s^p}}\ds
	\Big)^{2/(p-2)} \ge \Big( \sqrt{\frac{p}{2M}}\int_0^1 \frac{1}{\sqrt{1-s^p}}\ds
	\Big)^{2/(p-2)}.
\end{equation}
Since $\cD_{j,\ge0}^+$ is connected, for any $A>0$ smaller than the right hand side in \eqref{Gl Proof of Them part (iii)} there is a solution $(\mu_j,w_j)\in\cD_{j,\ge0}^+$ such that $\|w_j\|_\infty = A$. By definition of $\La_j^+$ we moreover have $\mu_j\le\La_j^+$. If a subsequence of $(\La_j^+)_j$ were bounded from above then so would a subsequence of $(\mu_j)_j$, and the Arzel\`{a}-Ascoli-Theorem would yield a
convergent subsequence of $(\mu_j,w_j)_j$. Since the number of nodes of $w_j$ tends to infinity as $j\to\infty$ the limit function would be a solution of \eqref{Gl eq ODE} having at least one zero of multiplicity two. By Proposition~\ref{Prop 2} this limit function would have to be trivial which contradicts $\|w_j\|_\infty=A>0$ for all $j\in\N_0$. Hence, the assumption was false and thus $\La_j^+\to\infty$ as $j\to\infty$.

\medskip

{\it Proof of (iv).}\; From the lower a-priori estimates from Lemma
\ref{Lem Apriori estimates I} we obtain
\begin{align} \label{Gl Proof of Thm part (iv) 1}
\limsup_{j\to\infty}\cD_{j,\le 0}^{+} = \emptyset.
\end{align}
Now assume for contradiction that there is $\eps>0$ and  $(\la_j,v_j)\in\cD_{j,\ge0}^+$ satisfying the inequalities $\eps\le \|v_j\|_\infty\le\eps^{-1}$ and $0\le \la_j\le \eps^{-1}$ for infinitely many $j\in\N_0$. Arguing as in the proof of (iii) we may use the Arzel\`{a}-Ascoli Theorem to find a uniformly converging subsequence which converges to the trivial solution which contradicts $\eps>0$. Hence, the assumption was false and we obtain
\begin{align} \label{Gl Proof of Thm part (iv) 2}
\limsup_{j\to\infty} \cD_{j,\ge0}^+
 \subset \bigcap_{\eps>0} \{ (\la,v) \in \R_{\ge 0}\times Y:
	(\la,\|v\|_\infty) \notin [0,\eps^{-1}]\times[\eps,\eps^{-1}] \}
	= \R_{\ge 0}\times\{0\}.
\end{align}
From \eqref{Gl Proof of Thm part (iv) 1} and \eqref{Gl Proof of Thm part (iv) 2} we get
\begin{align} \label{Gl Proof of Thm part (iv) 3}
\limsup_{j\to\infty} \cD_j^+ \subset \R_{\ge 0}\times\{0\}.
\end{align}
On the other hand (i) implies $(0,0)\in \liminf_{j\to\infty} \cD_j^+$ so that assertion (iv) is proved once we show $(\la,0)\in \liminf_{j\to\infty} \cD_j^+$ for all $\la>0$.  Indeed, once this is proved we have
$$
  \R_{\ge 0}\times\{0\} \subset \liminf_{j\to\infty} \cD_j^+
  \subset \limsup_{j\to\infty} \cD_j^+ \subset \R_{\ge 0}\times\{0\}
$$
which gives the result. Since the claim $(\la,0)\in \liminf_{j\to\infty} \cD_j^+$ is a direct consequence of (iv)(a) it remains to prove the claims (iv)(a) and (iv)(b).

{\it Proof of (iv)(a),(b):\;} Let $\la\in\R$ be fixed. According to (iii) there is a smallest number $j^+(\la)\in\N_0$ such that
$$
\La_j^+ > \la \qquad\text{for all }j\ge j^+(\la),
$$
in particular $j^+(\la)=0$ for all $\la\le 0$. From \eqref{Gl pr negative part} and the
definition of $\La_j^+$ we infer $\la\in(-\infty,\La_j^+]=\pr(\cD_j^+)$   so that we may define $\un{v}_j^+,\ov{v}_j^+$ to be the nontrivial solutions of \eqref{Gl eq ODE} in $\cD_j^+$ of least respectively largest maximum norm. These (not necessarily different) solutions exist due to our a-priori estimates and the Arzel\`{a}-Ascoli Theorem since minimizing and maximizing sequences of solutions of \eqref{Gl eq ODE} are equibounded away from zero or infinity. Now let us prove  $\|\ov{v}_j^+\|_\infty\to\infty$ as $j\to\infty$ for every $\la\in\R$ and $\|\un{v}_j^+\|_\infty\to 0$ as $j\to\infty$ for every $\la>0$. Observe that  $\|\un{v}_j^+\|_\infty\to0$ implies $\|\un{v}_j^+\|_{\cC^1}\to0$.

\medskip

So let $\eps>0$ be arbitrary. From \eqref{Gl Proof of Thm part (iv) 3} we infer that there is a natural number $\bar j^+\ge j^+(\la)$ depending on $\la$ and $\eps$ such that
$$
  (\|v\|_\infty,\mu) \notin
  [0,\eps^{-1}]\times [-|\la|,0] \cup [\eps,\eps^{-1}]\times (0,\la]
  \quad\text{for all }(v,\mu)\in\cD_j^+, j\ge \bar j^+,
$$
see Figure~\ref{Fig proof of (iv)}. In case $\la>0$ the connectedness of $\cD_j^+$ implies $\|\un{v}_j^+\|_\infty < \eps,\|\ov{v}_j^+\|_\infty>\eps^{-1}$ whereas in case $\la>0$ we have $\|\ov{v}_j^+\|_\infty > \eps^{-1}$. Since $\eps>0$ was chosen arbitrarily we obtain that $(\la,0)$ is a bifurcation point for \eqref{Gl eq ODE} for positive $\la$ and that $\la$ is a bifurcation point at infinity for all $\la\in\R$. Finally, Proposition~\ref{Prop Vorbereitung AprioriII 1} and 
Lemma~\ref{Lem Apriori estimates II} imply that $(\la,0)$ is not a branching point for \eqref{Gl eq ODE} since the number of nodal intervals is constant along continua away from the trivial solution. Similarly the a-priori estimates from 
Lemma~\ref{Lem Apriori estimates I} and~\ref{Lem Apriori estimates II} imply that there is no branching point at infinity at any given $\lambda\in\R$.

\medskip

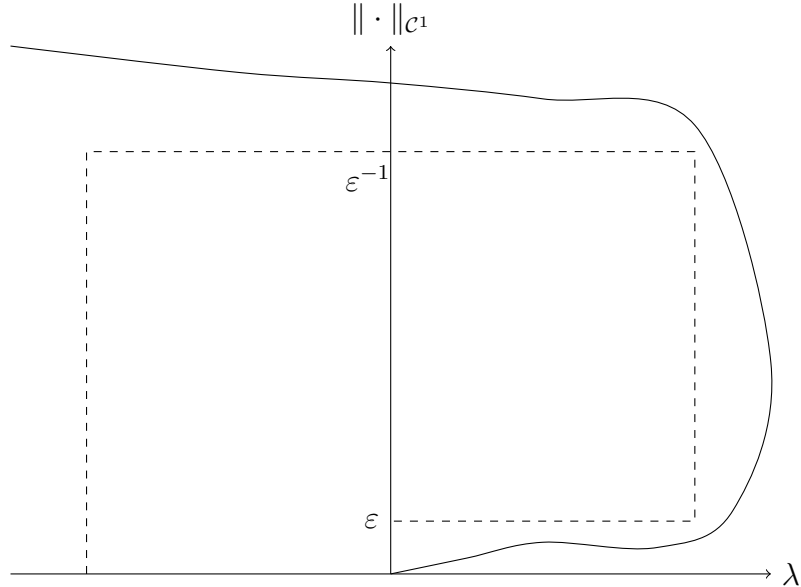
\begin{figure}[!htp]
\centering
\begin{tikzpicture}[yscale=0.7, xscale=0.5]
  \draw[->] (-10,0) -- (10,0) node[right] {$\lambda$};
  \draw[->] (0,0) -- (0,10) node[above] {$\|\cdot\|_{\cC^1}$};
  \draw [dashed] (-8,0) -- (-8,8) -- (8,8) -- (8,1) -- (0,1);
  \node at (-0.5,1) {$\eps$};
  \node at (-0.6,7.5) {$\eps^{-1}$};
  \draw plot[smooth, tension=0.6] coordinates { (0,0) (2,0.3)  (4,0.6) (7,0.5) (9,1.2)
  	  (10,4) (8,8.5) (4,9) (0,9.3) (-4,9.5) (-10, 10)};	    	
\end{tikzpicture}
  \caption{Illustration for the proof of part (iv)}
  \label{Fig proof of (iv)}
\end{figure}

\medskip

{\it Proof of (v):\;} Now assume that (B4) is satisfied. Then every solution $(\la,v)$ of \eqref{Gl eq ODE} satisfies
$$
\int_0^1 v'(r)^2\,dr = \int_0^1 h_\la(r,v(r))v(r)\,dr.
$$
In case $\la>0$ the solutions $(\la,\un{v}_j^+)$ from (iv)(a) exist for $j\ge j^+(\la)$. Using the above identity as well as $\|\un{v}_j^+\|_\infty\to 0$ as $j\to\infty$ we obtain from the second inequality in (B4)
\begin{align*}
I_\la(\un{v}_j^+)
 &= \frac12\int_0^1 (\un{v}_j^+)'(r)^2\,dr - \int_0^1 H_\la(r,\un{v}_j^+(r))\,dr \\
 &= \frac12\int_0^1 \Big( h_\la(r,\un{v}_j^+(r))\un{v}_j^+(r)
     - 2H_\la(r,\un{v}_j^+(r))\Big)\,dr \\
 &\to 0^-.
\end{align*}
Now let $\la\in\R$ be arbitrary so  that the solutions $(\la,\ov{v}_j^+)$ from (iv)(b) exist for $j\ge j^+(\la)$. In order to prove $I_\la(\ov{v}_j^+)\to \infty$ as $j\to\infty$ we use the estimate
\begin{align*}
\|\ov{v}_j^+\|_\infty^2
 &\le \int_0^1 |(\ov{v}_j^+)'(r)|^2\,dr \\
 &= (1-t) \int_0^1 h_\la(r,\ov{v}_j^+(r))\ov{v}_j^+(r)\,dr
     + t\cdot \Big(2 I_\la(\ov{v}_j^+) + 2\int_0^1 H_\la(r,\ov{v}_j^+(r))\,dr \Big) \\
 &= 2tI_\la(\ov{v}_j^+) + \int_0^1 2t H_\la(r,\ov{v}_j^+(r))
     + (1-t)h_\la(r,\ov{v}_j^+(r))\ov{v}_j^+(r)\,dr
\end{align*}
for all $t\in\R$. From the first inequality in assumption (B4) we obtain
$h_\la(r,z)z\ge \mu H_\la(r,z)$ for some $\mu>2$ and all $r\in [0,1]$, $|z|\ge z_0$. Using the above estimate for $t=\tfrac{\mu}{\mu-2}$ we obtain
$2t H_\la(r,z)+ (1-t) h_\la(r,z)z\le 0$ for all $r\in [0,1],|z|\ge z_0$ and thus
$$
\|\ov{v}_j^+\|_\infty^2 \le 2tI_\la(\ov{v}_j^+) + C
 \qquad\text{where } C:= \max_{|z|\le z_0,r\in [0,1]} 2t H_\la(r,z)+ (1-t) h_\la(r,z)z
$$
From $\|\ov{v}_j^+\|_\infty\to \infty$ as $j\to\infty$ we obtain the result.
\end{altproof}

\medskip

\begin{altproof}{Theorem \ref{Thm 1}}
Let the assumptions of Theorem \ref{Thm 1} hold, so that $f_\la$ is a nonlinearity satisfying the assumptions (A1), (A2), (A3) for $1<q<2<p<\infty$ and $m_1,M_1>0$ and $K_1(\la,s)$ as required. Let then $h_\la$ be defined as in \eqref{Gl Definition h}, i.e.
$$
h_\la(r,z,\xi) = \phi'(r)^2f_\la(\phi(r),z,\xi/\phi'(r)).
$$
Then $h_\la$ satisfies (B1), (B2), (B3) where $m,M,K(\la,s)$ are chosen as in \eqref{Gl def mMKc}. Indeed, the estimate $m_2\le \phi'(r)\le M_2$ for all $r\in [0,1]$ from Proposition~\ref{Prop 1}~(ii) yields the following inequality for $r\in [0,1]$, 
$0\leq |z|,\xi\le s$:
\begin{align*}
|\pa_rh_\la(r,z,\xi)|
 &= \Big| 2\phi'(r)\phi''(r) f_\la(\phi(r),z,\xi/\phi'(r))
     + \phi'(r)^3\pa_rf_\la(\phi(r),z,\xi/\phi'(r))\\
 &\hspace{1cm} - \xi\phi''(r)\pa_\xi f_\la(\phi(r),z,\xi/\phi'(r)) \Big| \\
 &\le (2|\phi'(r)||\phi''(r)|+(|\phi'(r)|^3
     + s|\phi''(r)|)K_1(\la,s/m_2)) \cdot |f_\la(\phi(r),z,\xi/\phi'(r))|\\
 &\le (2\|\phi''(\phi')^{-1}\|_\infty+(\|\phi'\|_\infty
     + s\|\phi''(\phi')^{-2}\|_\infty)K_1(\la,s/m_2)) \cdot |h_\la(r,z,\xi)|\\
 &\le (2m_2^{-1}\|\phi''\|_\infty
     + (M_2+sm_2^{-2}\|\phi''\|_\infty)K_1(\la,s/m_2)) \cdot |h_\la(r,z,\xi)|\\
 &\le K(\la,s) |h_\la(r,z,\xi)|,
\end{align*}
and
\begin{align*}
|\pa_\xi h_\la(r,z,\xi)|
 &\le |\phi'(r)| |\pa_\xi f_\la(\phi(r),z,\xi/\phi'(r))| \\
 &\le |\phi'(r)| K_1(\la,s/m_2) \cdot  |f_\la(\phi(r),z,\xi/\phi'(r)))| \\
 &\le m_2^{-1} K_1(\la,s/m_2) \cdot  |h_\la(r,z,\xi)| \\
 &\le K(\la,s)  |h_\la(r,z,\xi)|.
\end{align*}
Moreover, if $f_\la$ satisfies (A4) then $h_\la$ satisfies (B4). By Theorem~\ref{Thm 2} there are solution continua $\cD_j^\pm$ enjoying the properties (i)-(v) from Theorem \ref{Thm 2}. We set
$$
\cC_j^\pm
 := \{ (\la,(v\circ\phi^{-1})(|\cdot|))\in\R\times X : (\la,v)\in \cD_j^\pm\}
$$
where $\phi:[0,1]\to [\rho_1,\rho_2]$ is the diffeomorphism from \eqref{Gl Def phi}. Then
Proposition~\ref{Prop 1} implies that $\cC_j^\pm$ consists of solutions of \eqref{Gl eq ODE} having precisely
$j+1$ interior nodal annuli $A_0,\ldots,A_j$ with sign $\pm(-1)^k$ on $A_k$. The claims (i)-(iv) from
Theorem~\ref{Thm 1} follow directly from the corresponding statements in Theorem~\ref{Thm 2} and
$\|v\circ\phi^{-1}(|\cdot|)\|_{L^\infty(\Om)}= \|v\|_{\infty}$. The proof of claim (v) is, up to textual
modifications, the same as in Theorem~\ref{Thm 1} so that the proof is finished.
\end{altproof}

\section{Appendix A - The one-dimensional Ambrosetti-Brezis-Cerami problem} \label{sec 4}
In this section we present the time map analysis which allows to find all nontrivial solutions of the one-dimensional Ambrosetti-Brezis-Cerami problem
\begin{equation}\label{Gl 1D ABC problem}
- u'' = \la |u|^{q-2}u + |u|^{p-2}u = g_\la(u) \qquad\text{in }(0,1),\qquad  u(0)=u(1)=0,
\end{equation}
where $1<q<2<p<\infty$. As before we set $G_\la(z)=\frac{\la}{q}|z|^q+\frac{1}{p}|z|^p$ for $z\in\R$ and
\[
T_\la(\al) = \int_0^\al \frac{1}{\sqrt{2(G_\la(\al)-G_\la(s)})}\ds.
\]
We recall that a subinterval $I\subset [0,1]$ is called a nodal interval of a solution $u$ of \eqref{Gl 1D ABC problem} if $|u|$ is positive on $\mathring{I}$ and vanishes identically on $\pa I$. A first step towards a complete picture of all nontrivial solutions of \eqref{Gl 1D ABC problem} is the following result.

\begin{prop}\label{Prop App A 1}
Let $\la\in\R$ and let $u$ be a solution of \eqref{Gl 1D ABC problem} with nodal interval $[a,b]\subset [0,1]$. Then $(a,b)$ contains precisely one critical point. It is given by $\xi=\frac{a+b}{2}$ and we have $u(\xi+t)=u(\xi-t)$ for $|t|\le\frac{b-a}{2}$.
\end{prop}

\begin{proof}
In order to prove the first claim we show that every critical point $\xi$ of $u$ is a local maximum in case $u(\xi)>0$ and a local minimum in case $u(\xi)<0$. Indeed, multiplying the differential equation \eqref{Gl 1D ABC problem} with $2u'$ and integrating the resulting equation from $a$ to $\xi$ gives
$$
  u'(a)^2 = \frac{2\la}{q}|u(\xi)|^q+\frac{2}{p}|u(\xi)|^p
$$
where we have used $u'(\xi)=u(a)=0$. Hence, the first claim follows from
$$
-u''(\xi)u(\xi)
 = \la |u(\xi)|^q + |u(\xi)|^p
 = \frac{q}{2}u'(a)^2 + \frac{p-q}{p}|u(\xi)|^p
 > 0.
$$
Moreover, since $g_\la$ is locally Lipschitz-continuous on $\R\sm\{0\}$ we obtain that the initial value problem at $\xi$ has a unique solution which implies $u(\xi+t)=u(\xi-t)$ for $a\le \xi-t,\xi+t\le b$. From $u(a)=u(b)=0$ and $|u|>0$ on $(a,b)$ we infer $\xi=\frac{a+b}{2}$ which gives the result.
\end{proof}

Now let us determine all solutions of \eqref{Gl 1D ABC problem} with a given nodal interval $[a,b]\subset[0,1]$. The next Proposition shows that every solution $\al>0$ of the scalar equation $T_\la(\al) = \tfrac{b-a}{2}$ generates precisely two such solutions which have maximum norm $\al$. To this end we introduce the function $\phi_{\al,\la}$ given by
\begin{equation} \label{Gl def phi}
\phi_{\al,\la}(z)
 = \int_0^z \frac{1}{\sqrt{2(G_\la(\al)-G_\la(s))}}\ds
 \qquad\text{for } 0\le z\le \al
\end{equation}
which is well-defined for $\al\ge (\tfrac{p\la_-}{q})^{1/(p-q)}$ where $\la_-=\max\{0,-\la\}$. Notice that in case  $\al<(\tfrac{p\la_-}{q})^{1/(p-q)}$ the argument of the square root appearing in \eqref{Gl def phi} attains negative values. The critical case $\alpha=(\tfrac{p\la_-}{q})^{1/(p-q)}$ will later play a central role in the analysis of dead-core solutions.

\begin{prop}\label{Prop App A 2}
Let $\la\in\R,[a,b]\subset [0,1]$ and $\al>0$. There is a solution $u$ of \eqref{Gl 1D ABC problem} with nodal interval $[a,b]$ and $\|u\|_\infty=\al$ if and only if
\begin{equation}\label{Gl Timemap}
\al\ge \big(\frac{p\la_-}{q}\big)^{1/(p-q)}
\quad\text{and}\quad
T_\la(\al) = \frac{b-a}{2}.
\end{equation}
In this case every such solution is given by $u=\pm w_{\al,\la}$ where
$$
w_{\al,\la}(x) := \phi_{\al,\la}^{-1}(x-a)\quad\text{ for }a\le x\le \tfrac{a+b}{2},
\qquad
w_{\al,\la}(x) := \phi_{\al,\la}^{-1}(b-x)\quad\text{ for }\tfrac{a+b}{2}\le x\le b.
$$
In particular we have $u'(a)=0$ if and only if $\al=(\tfrac{p\la_-}{q})^{1/(p-q)}$.
\end{prop}

\begin{proof}
Let first $u$ be a solution of \eqref{Gl 1D ABC problem} with nodal interval $[a,b]$ and   $\|u\|_\infty=\al$. We may assume $u>0$ on $(a,b)$ so that Proposition~\ref{Prop App A 1} yields $u'>0$ on $(a,\tfrac{a+b}{2})$ and ${u(\tfrac{a+b}{2})=\al}$. Multiplying \eqref{Gl 1D ABC problem} with $2u'$ and integrating the resulting equation from $x$ to $\tfrac{a+b}{2}$ gives
\begin{equation} \label{Gl 1st order eq}
u'(x)^2 = 2(G_\la(\al)-G_\la(u(x)))\qquad\text{for }  x\in\Big(a,\frac{a+b}{2}\Big).
\end{equation}
Hence, Proposition \ref{Prop App A 1} implies $G_\la(\al)>G_\la(z)$ for all $z\in(0,\al)$ and thus $\al\ge (\frac{p\la_-}{q})^{1/(p-q)}$. Moreover, 
\eqref{Gl 1st order eq} gives
\begin{align*}
\phi_{\al,\la}(u(x))
 = \int_a^x \frac{u'(t)}{\sqrt{2(G_\la(\al)-G_\la(u(t)))}}\,dt
 = \int_a^x 1 \,dt
 = x-a  \qquad\text{for }x\in\Big(a,\frac{a+b}{2}\Big)
\end{align*}
and we obtain the solution formula $u= w_{\al,\la}$. Furthermore, \eqref{Gl Timemap} follows from
\begin{align*}
T_\la(\al)
 = \int_0^\al \frac{1}{\sqrt{2(G_\la(\al)-G_\la(s)})}\ds
 = \phi_{\al,\la}(\al)
 = \phi_{\al,\la}\Big(u\Big(\frac{a+b}{2}\Big)\Big)
 = \frac{b-a}{2}
\end{align*}
and \eqref{Gl 1st order eq} implies that we have $u'(a)=0$ if and only if $G_\la(\al)=0$  which is equivalent to $\al=(\tfrac{p\la_-}{q})^{1/(p-q)}$. Vice versa, if $\al$ is a solution of \eqref{Gl Timemap} then $\pm w_{\al,\la}$ is a solution of \eqref{Gl 1D ABC problem} with maximum norm $\al$ and nodal interval $[a,b]$. This finishes the proof.
\end{proof}

The next step is to investigate how these solutions can be patched together in order to find solutions of \eqref{Gl 1D ABC problem} by solving the initial value problems at the boundary of each nodal interval. Looking for solutions with precisely $j+1$ nodal intervals the following threshold value plays a significant role:
\begin{equation} \label{Gl defn lambdajunten*}
\la_{j*}
 := - \Big( \sqrt{2q} \big(\frac{p}{q}\big)^{(2-q)/2(p-q)}
    \int_0^1 \frac{1}{\sqrt{t^q-t^p}}\,dt \Big)^{2(p-q)/(p-2)}
    	\cdot (j+1)^{2(p-q)/(p-2)}.
\end{equation}
In the next Proposition we show that for $\la>\la_{j*}$ solutions with $j+1$ nodal intervals have precisely $j$ interior nodes located at $\tfrac{1}{j+1},\ldots,\tfrac{j}{j+1}$ and that each solution is pointwise symmetric with respect to all of its nodes. When $\la$ tends to $\la_{j*}$ from the right the slopes at the zeros tend to 0 and dead core solutions appear for $\la<\la_{j*}$. We show that the set of all dead core solutions of \eqref{Gl 1D ABC problem} for a given $\la<\la_{j*}$  can be described by $j+1$ discrete parameters $\sigma_0,\ldots,\sigma_j\in\{-1,+1\}$ and $j+1$ continuous parameters $a_0,\ldots,a_j$ belonging to
$$
\cZ_{j,\la}
 := \{ a\in [0,1]^{j+1}:
       0\le a_0\le a_0+l(\la)\le a_1\le a_1+l(\la) \le \ldots \le a_j\le
        a_j+l(\la)\le 1\}.
$$
where $l(\la)$ is the length of the nodal interval of an arbitrary dead core solution given by
\begin{equation} \label{Def lj}
l(\la)
 := 2T_\la\Big( \Big(\frac{p\la_-}{q}\Big)^{1/(p-q)}\Big)
 = \frac{1}{j+1} \Big(\frac{|\la_{j*}|}{|\la|}\Big)^{(p-2)/2(p-q)}.
\end{equation}
For a verification of the latter equality one uses Proposition \ref{Prop properties time map} (v) (to be proved later). The following Proposition proves that the nontrivial solutions of \eqref{Gl 1D ABC problem} with  precisely $j+1$ nodal intervals in $[0,1]$ look like the functions in Figure \ref{Fig normal and deadcore solutions}.

\medskip

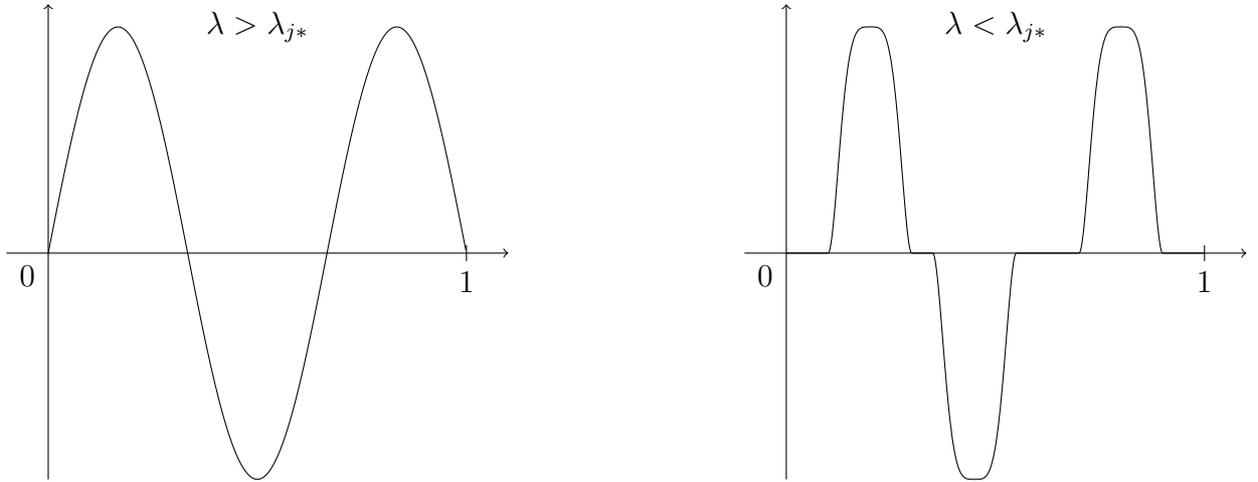
\begin{figure}[!htp]
\centering
\subfigure
  {
  \begin{tikzpicture}[domain=0:1, xscale=5.5,yscale=3,samples=100]
      \draw[->] (-0.1,0)  -- (1.1,0);
      \draw[->] (0,-1) -- (0,1.1);
      \node at (-0.05,-0.1) {$0$};
      \node at (0.5,1) {$\lambda>\lambda_{j*}$};
      \draw (1,1pt) -- (1,-1pt)  node[right,below]{$1$};
      \draw[smooth] plot (\x,{cos((3*pi*\x - pi/2) r)});
      \label{pic 2a}
  \end{tikzpicture}
  }
  \hfill
  \subfigure
  {
  \begin{tikzpicture}[samples=100, xscale=5.5,yscale=3]
      \draw[->] (-0.1,0)  -- (1.1,0);
      \draw[->] (0,-1) -- (0,1.1);
      \node at (-0.05,-0.1) {$0$};
      \node at (0.5,1) {$\lambda<\lambda_{j*}$};
      \draw (1,1pt) -- (1,-1pt) node[right,below]{$1$} ;
      \draw[smooth,domain=0:0.1] plot (\x,0);
      \draw[smooth,domain=0.1:0.3] plot (\x,{cos((50*pi*(\x-0.2)^2) r)^2});
      \draw[smooth,domain=0.3:0.35] plot (\x,0);
      \draw[smooth,domain=0.35:0.55] plot (\x,{-1*(cos((50*pi*(\x-0.45)^2) r))^2});
      \draw[smooth,domain=0.55:0.7] plot (\x,0);
      \draw[smooth,domain=0.7:0.9] plot (\x,{cos((50*pi*(\x-0.8)^2) r)^2});%
      \draw[smooth,domain=0.9:1] plot (\x,0);
      \label{pic 2b}
  \end{tikzpicture}
  }
\caption{Qualitative plots of the solutions}
 \label{Fig normal and deadcore solutions}
\end{figure}

\medskip

\begin{prop}\label{Prop App A 3}
Let $j\in\N_0,\la\in\R$ and let $u$ be a solution of \eqref{Gl 1D ABC problem} with
precisely $j+1$ nodal intervals in $[0,1]$, set $\al:=\|u\|_\infty$. Then the following holds:
\begin{itemize}
\item[(i)] In case $\la>\la_{j*}$ we have $T_\la(\al)=\tfrac{1}{2j+2}$ and
$$
u(x) = \pm (-1)^k w_{\al,\la}\Big(x-\frac{k}{j+1}\Big)
 \qquad\text{for all }x\in\Big[\frac{k}{j+1},\frac{k+1}{j+1}\Big]
 \text{ and } k=0,\ldots,j.
$$
\item[(ii)] In case $\la=\la_{j*}$ we have $\al=(\tfrac{p|\la_{j*}|}{q})^{1/(p-q)}$
and there are $\sigma_0,\ldots,\sigma_j\in\{-1,+1\}$ such that
$$
u(x) = \sigma_k w_{\al,\la}\Big(x-\frac{k}{j+1}\Big)
 \qquad\text{for all }x\in\Big[\frac{k}{j+1},\frac{k+1}{j+1}\Big]
 \text{ and } k=0,\ldots,j.
$$
\item[(iii)] In case $\la<\la_{j*}$ we have $\al = (\tfrac{p|\la|}{q})^{1/(p-q)}$ and
there are $\sigma_0,\ldots,\sigma_j\in\{-1,+1\}$ and $(a_0,\ldots,a_j)\in\cZ_{j,\la}$ such that
$$
u\equiv 0 \quad\text{on }
  [0,a_0]\cup[a_0+l(\la),a_1]\cup \ldots\cup[a_{j-1}+l(\la),a_j]\cup[a_j+l(\la),1]
$$
and
$$
u(x) = \sigma_k w_{\al,\la}(x-a_k) \qquad
        \text{for }x\in    [a_k,a_k+l(\la)] \text{ and all }  k=0,\ldots,j.
$$
\end{itemize}
\end{prop}

\begin{proof}
The proof of this result is accomplished in the following way. Given a solution $u$ of
\eqref{Gl 1D ABC problem} with precisely $j+1$ nodal intervals we show that
\begin{itemize}
  \item[(a)] $u'(0)\ne 0$ implies $\la>\la_{j*}$ and $u$ is given by the formula from (i) and
  \item[(b)] $u'(0)=0$ implies $\la\leq \la_{j*}$ and $u$ is given by the formulas from (ii) or (iii) according to $\la=\la_{j*}$ or $\la<\la_{j*}$.
\end{itemize}

In case $u'(0)\ne 0$ let $l>0$ denote the first positive zero of $u$. Then Proposition \ref{Prop App A 2} gives $u=\pm w_{\al,\la}$ on $[0,l]$ and $T_\la(\al) = \frac{l}{2}$ for some $\al>(\tfrac{p\la_-}{q})^{1/(p-q)}$. Using the symmetry of $w_{\al,\la}$ we get $u'(l)=-u'(0)$ and thus $u=\mp w_{\al,\la}(\cdot-l)$ on $[l,2l]$ again by 
Proposition~\ref{Prop App A 2}. Inductively we obtain 
$u=\pm(-1)^k w_{\al,\la}(\cdot-kl)$ on $[kl,(k+1)l]$ for all $k\in\N_0$. Since $u$ satisfies $u(1)=0$ and has precisely $j+1$ nodal intervals we obtain $l=\frac{1}{j+1}$ and thus $T_\la(\al)=\frac{l}{2}=\frac{1}{2j+2}$. Given that this equation has a solution $\al>(\frac{p\la_-}{q})^{1/(p-q)}$ we infer  $\la>\la_{j*}$ from
Proposition~\ref{Prop properties time map}~(v).

In case $u'(0)=0$ the parameter $\la$ must be negative by Proposition \ref{Prop 2}. Either the function $u$ vanishes identically on some right-sided neighbourhood of $0$ or $|u|$ is positive on a right-sided neighbourhood of $0$. Indeed, if there is a sequence $(x_n)$ converging to $0$ with $u(x_n)=0$ then continuity of $u$ implies $g_\la(u(t))u(t)<0$ for all $t\in [0,x_n]$ for sufficiently large $n$ and thus
\begin{align*}
0 = u'(x_n)u(x_n)
  = \int_0^{x_n} u'(t)^2 +   u''(t)u(t)\dt
  = \int_0^{x_n} u'(t)^2  - g_\la(u(t))u(t) \dt
  \ge \int_0^{x_n} u'(t)^2 \dt
\end{align*}
which implies $u\equiv 0$ on $[0,x_n]$. Here we used $g_\la(z)=\la|z|^{q-2}z+|z|^{p-2}z$ and that $\la$ is negative. Therefore, defining
$a_0:= \max\{ x\in [0,1]: u(t)=0    \text{ for all }t\in [0,x]\}$
we obtain $u'(a_0)=0$ and that $|u|$ is positive on some right-sided neighbourhood of $a_0$. Proposition
\ref{Prop App A 2} then implies $u=\pm w_{\al,\la}(\cdot-a_0)$ on
$[a_0,a_0+l(\la)]$ for $\al=(\tfrac{p\la-}{q})^{1/(p-q)}$, in particular $u(a_0+l(\la))=u'(a_0+l(\la))=0$.
Defining $a_1:= \max\{ x\in [a_0+l(\la),1]: u(t)=0 \text{ for all }t\in [a_0+l(\la),x]\}$ we obtain that $u$ vanishes on $[a_0+l(\la),a_1]$ and, again using
Proposition~\ref{Prop App A 2}, $u=\pm w_{\al,\la}(\cdot-a_1)$ on $[a_1,a_1+l(\la)]$. Repeating this process and using $u(1)=0$ and that $u$ has precisely $j+1$ nodal intervals we obtain $(j+1)l(\la)\le 1$ and thus $\la\le\la_{j*}$, see \eqref{Def lj}. From this we obtain claim (ii) and (iii) since $\la=\la_{j*}$ and \eqref{Def lj} imply $(j+1)l(\la)=1$ which in turn implies $a_k=a_{k-1}+l(\la)=a_{k-1}+\tfrac{1}{j+1}$ for
all $k=1,\ldots,j$ and thus $a_k=\frac{k}{j+1}$ for all $k=0,\ldots,j$.
\end{proof}

The above Proposition reduces the problem of finding all solutions of the boundary value problem \eqref{Gl 1D ABC problem} to the task of solving the scalar equation $T_\la(\al)=\frac{1}{2j+2}$ for $\la>\la_{j*}$. Hence, the solution theory for \eqref{Gl 1D ABC problem} depends on the properties of $T_\la$ which we list in the following Proposition. Its proof will be given in Appendix B.

\begin{prop} \label{Prop properties time map}
\begin{itemize}
\item[(i)] For all $\la\ge 0$ and all $\al>0$ the following estimates hold true:
\begin{align*}
T_\la(\al)
 &\le \Big( \frac{pq\al^{2-q}}{2\la p + 2q\al^{p-q}} \Big)^{1/2}
       \int_0^1 \frac{1}{\sqrt{1-s^q}}\ds, \\
T_\la(\al)
 &\ge \Big( \frac{pq\al^{2-q}}{2\la p + 2q\al^{p-q}} \Big)^{1/2} \int_0^1
        \frac{1}{\sqrt{1-s^p}}\ds.
 \end{align*}
\item[(ii)] For all $\la>0$ there is a uniquely determined number $\al_\la>0$ such
that $T_\la$ is strictly increasing on $(0,\al_\la]$ and strictly decreasing on $[\al_\la,\infty)$. Moreover, we have
$$
\lim_{\al\to 0^+} T_\la(\al) = \lim_{\al\to \infty} T_\la(\al) = 0.
$$
\item[(iii)] There are positive numbers $c_3,C_3$ such that for all $\la>0$ we have
$$
c_3 \la^{1/(p-q)} < \al_\la < C_3 \la^{1/(p-q)},\qquad
c_3 \la^{(2-p)/2(p-q)} <T_\la(\al_\la)< C_3  \la^{(2-p)/2(p-q)}.
$$
\item[(iv)] The map $\la\mapsto T_\la(\al_\la)$ is decreasing on $\R_{>0}$ and there are uniquely determined positive numbers $\Lambda_0<\Lambda_1<\Lambda_2<\ldots\to\infty$ with $T_{\Lambda_j}(\al_{\Lambda_j}) = \frac{1}{2j+2}$ for all $j\in\N_0$. There are positive numbers $c_4,C_4$ such that the following estimates hold for all $j\in\N_0$:
\begin{align*}
c_4 (j+1)^{2(p-q)//p-2)} \le \Lambda_j \le C_4 (j+1)^{2(p-q)/(p-2)}.
\end{align*}
\item[(v)] For all $\la<0$ the function $T_\la$ is well-defined and decreasing on $[(\tfrac{p|\la|}{q})^{1/(p-q)},\infty)$. Moreover, we have
$$
T_\la\Big( \Big(\frac{p|\la|}{q}\Big)^{1/(p-q)}  \Big)
 = \frac{1}{2j+2} \Big(\frac{|\la_{j*}|}{|\la|}\Big)^{(p-2)/2(p-q)}, \qquad
    \lim_{\al\to \infty} T_\la(\al) = 0.
$$
\end{itemize}
\end{prop}

Using the properties of $T_\la$ described in the parts (ii) and (iv) of the previous
Proposition we obtain complete information about the solutions of the equation $T_\la(\al) = \frac{1}{2j+2}$. The above result are illustrated in the pictures 
\ref{pic Tlambda lambdageq0} and \ref{pic Tlambda lambda<0}.

\medskip

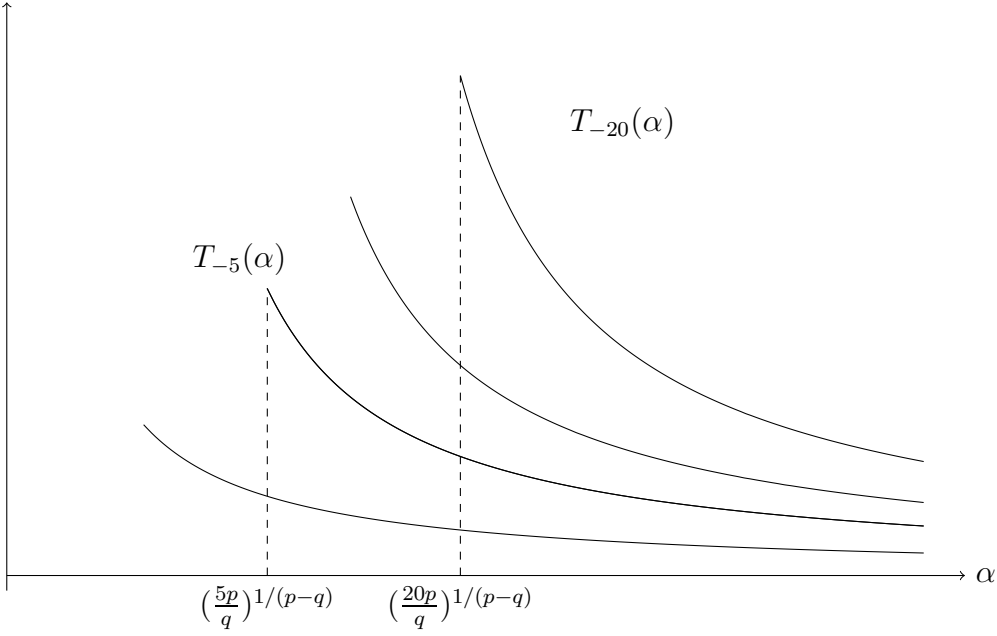
\begin{figure}[!htp]
 \centering
    \begin{tikzpicture}[xscale=1.8, yscale=2, samples=100]
      \draw[->] (0,0) -- (7,0) node[right]{$\al$};
      \draw[->] (0,-0.1) -- (0,3.8);
      \draw[smooth,domain=0.2^0.4:6.7] plot (\x,{(0.2^0.4/(1-0.2^0.4+\x))}); 
      \draw[smooth,domain=1:6.7] plot (\x,{(1^0.4/(1-1^0.4+\x))}); 
      \draw[smooth,domain=5^0.4:6.7] plot (\x,{(5^0.4/(1-5^0.4+\x))}); 
      \draw[smooth,domain=10^0.4:6.7] plot (\x,{(10^0.4/(1-10^0.4+\x))}); 
      \draw[smooth,domain=20^0.4:6.7] plot (\x,{(20^0.4/(1-20^0.4+\x))}); 
      \draw[dashed] (5^0.4,0) node[below] {$(\frac{5p}{q})^{1/(p-q)}$} -- (5^0.4,5^0.4);
      \draw[dashed] (20^0.4,0)  node[below] {$(\frac{20p}{q})^{1/(p-q)}$}  -- (20^0.4,20^0.4);
      \node at (4.5,3) {$T_{-20}(\al)$};
      \node at (1.7,2.1) {$T_{-5}(\al)$};
      \label{pic b}
    \end{tikzpicture}
 \caption{Qualitative plots of $T_\la$ for $q=1.5,p=4$ and 
 $\la\in\{-0.2,-1,-5,-10,-20\}$.
 }
 \label{pic Tlambda lambda<0}
\end{figure}

\medskip

For all $\la\in (0,\La_j)$ the equation has exactly two different solutions $\un{\al}_j(\la)\in (0,\al_\la)$ and $\ov{\al}_j(\la)\in (\al_\la,\infty)$ giving rise to exactly four different solutions $\un{u}_j(\la),-\un{u}_j(\la),\ov{u}_j(\la),-\ov{u}_j(\la)$ with $j$ interior nodes and they are given by the formulas from Proposition \ref{Prop App A 3} (i) for $\al=\un{\al}_j(\la)$ respectively $\al=\ov{\al}_j(\la)$. As $\la$ tends to $\La_j$ from the left the solutions $\un{u}_j(\la),\ov{u}_j(\la)$ merge into each other as both values $\un{\al}_j(\la),\ov{\al}_j(\la)$ converge to $\al_{\La_j}$. In case $\la>\La_j$ Proposition \ref{Prop properties time map} (iv) implies $T_\la(\al)<\tfrac{1}{2j+2}$ for all $\al>0$ so that no solutions with $j+1$ nodal intervals exist according to Proposition \ref{Prop App A 3} (i). As $\la$ tends to 0 from the right we observe $\un{\al}_j(\la)\to 0$ so that $\un{u}_j(\la)$ converges to the trivial solution while $\ov{u}_j(\la)$ converges to the uniquely determined nontrivial solution of $-u''=|u|^{p-2}u,u(0)=u(1)=0,u'(0)>0$ with $j$ interior nodes. The solutions $\ov{u}_j(\la),-\ov{u}_j(\la)$ persist in the range  $\la_{j*}\le \la\le 0$ if now, for $\la\le 0$, the value $\ov{\al}_j(\la)\ge(\tfrac{p|\la|}{q})^{1/(p-q)}$ denotes the unique solution of $T_\la(\al)=\frac{1}{2j+2}$, see Proposition \ref{Prop properties time map} (v). As $\la$ tends to $\la_{j*}$ from the right we observe $\ov{\al}_j(\la)\to(\tfrac{p|\la_{j*}|}{q})^{1/(p-q)}$ and that the slopes at the zeros $0,\tfrac{1}{j+1},\ldots,\tfrac{j}{j+1},1$ tend to 0 so that there is a continuous transition to the dead core solutions described in Proposition~\ref{Prop App A 3}~(ii). As a consequence the solution continua $\cC_j^\pm$ from Theorem~\ref{Thm 1} in the special case $n=1$ and $h_\la(r,z,\xi)=g_\la(z)$ are given by the following theorem.

\begin{thm}\label{Thm 1D case}
 Let $j\in\N_0$. Then all nontrivial solutions of \eqref{Gl 1D ABC problem} with $j+1$ nodal intervals $I_0,\ldots,I_j$ and sign $\pm (-1)^k$ on $I_k$ are given by
 $\cC_j^+ = \cC_{j,1}^+ \cup \cC_{j,2}^+$ and
 $\cC_j^- = \cC_{j,1}^- \cup \cC_{j,2}^-$ where
 \begin{align*}
  \cC_{j,1}^\pm
   &= \{ (\pm\un{u}_j(\la),\la):0\le \la\le \La_j\}
      \cup \{ (\pm\ov{u}_j(\la),\la): \la_{j*}<\la<\La_j\}, \\
  \cC_{j,2}^\pm
   &= \{ (\pm u_j(\la,a,\sigma),\la): \la<\la_{j*},a\in
      \cZ_{j,\la},\sigma=(1,-1,1,\ldots,(-1)^{j+1})\}
 \end{align*}
 and $u_j(\la,a,\sigma)$ denotes the dead core solution given by 
 Proposition~\ref{Prop App A 3} (iii).
\end{thm}

We finally remark that for all $k\in\{0,\ldots,j\}$ and $\la\le \la_{j*}$ there are solutions with $j$ interior (degenerate) zeros and only $k$ sign changes on $[0,1]$. These solutions are given by $u_j(\la,a,\sigma)$ for $a\in\cZ_{j,\la}$ and vectors $\sigma\in \{-1,+1\}^{j+1}$ which satisfy $\sigma_i\sigma_{i+1}=-1$ for precisely $k$ different indices in $i$.

\section{Appendix B - Proof of Proposition \ref{Prop Vorbereitung AprioriII 1}, Proposition \ref{Prop Vorbereitung AprioriII 2}, and 
Proposition~\ref{Prop properties time map}} \label{sec 5}
In this section we provide the proofs of some technical results concerning the time map $T_\la$ which we defined in \eqref{Gl def timemap}. Let us first mention that Proposition~\ref{Prop Vorbereitung AprioriII 1} is entirely contained in Proposition~\ref{Prop properties time map}. We will use the following equation
\begin{align}\label{Gl eq timemap}
T_\la(\al)
 &= \int_0^\al\frac{1}{\sqrt{2(G_\la(\al)-G_\la(z)})}\,dz \notag \\
 &= \int_0^\al \frac{1}{\sqrt{\frac{2\la}{q}(\al^q-z^q) +
       \frac{2}{p}(\al^p-z^p)}} \,dz       \notag \\
 &= \int_0^1 \Big(\frac{pq\al^{2-q}}{2p\la(1-s^q)+2q\al^{p-q}(1-s^p)}\Big)^{1/2}
      \,ds.
\end{align}

\medskip

\begin{altproof}{Proposition \ref{Prop properties time map} (i)}
From the inequality $1-s^q\le 1-s^p$ for all $s\in [0,1]$ we obtain
\[
T_\la(\al)
 = \int_0^1 \Big(\frac{pq\al^{2-q}}{2\la p (1-s^q)+2 q
      \al^{p-q}(1-s^p)}\Big)^{1/2} \,ds
 \le \Big( \frac{pq\al^{2-q}}{2\la p + 2 q\al^{p-q}}\Big)^{1/2}
        \int_0^1\frac{1}{\sqrt{1-s^q}}\,ds,
\]
and
\[
T_\la(\al)
 = \int_0^1 \Big(\frac{pq\al^{2-q}}{2\la p(1-s^q)+2q\al^{p-q}(1-s^p)}\Big)^{1/2}
     \,ds
 \ge \Big( \frac{pq\al^{2-q}}{2\la p + 2q\al^{p-q}}\Big)^{1/2}
        \int_0^1\frac{1}{\sqrt{1-s^p}}\,ds
\]
and assertion (i) follows.
\end{altproof}

\medskip

\begin{altproof}{Proposition \ref{Prop properties time map} (ii)}
The existence of at least one critical point of $T_\la$ follows from the intermediate value theorem since the formula
\begin{equation} \label{Gl Tlambdap}
T_\la'(\al)
 = \sqrt{pq}\al^{-q/2} \int_0^1 \frac{\la p(2-q)(1-s^q)-
       q(p-2)\al^{p-q}(1-s^p)  }{\big(2\la p(1-s^q)+2q\al^{p-q}(1-s^p)\big)^{3/2}}  \,ds
 \end{equation}
implies that $T_\la'(\al)\al^{q/2}$ tends to a positive value as $\al\to 0^+$ and $T_\la'(\al)\al^{p/2}$ tends to a negative value as $\al\to\infty$. Having proved the existence of a critical point of $T_\la$ it remains to prove uniqueness. To this end we prove $T_\la''(\al)<0$ for all $\al>0$ satisfying $T_\la'(\al)=0$.

Every critical point $\al$ of $T_\la$ satisfies
\begin{align*}
0 &= T'_\la(\al) \cdot (\sqrt{pq}\al^{-q/2})^{-1}  \\
  &= \int_0^1 \frac{\la p(2-q)(1-s^q) - q(p-2)\al^{p-q}(1-s^p)}
                   {\big(2\la p(1-s^q)+2q\al^{p-q}(1-s^p)\big)^{3/2}}
         \,ds \\
  &= \int_0^1 \frac{\la p(2-q)(1-s^q)}
                   {\big(2\la p(1-s^q)+2q\al^{p-q}(1-s^p)\big)^{3/2}}
                    \cdot \Big(1-\frac{q(p-2)(1-s^p)}{\la p(2-q)(1-s^q)}\al^{p-q}\Big)
         \,ds
\end{align*}
Hence, the second factor in the above integral must change sign which implies
\begin{equation} \label{Gl Tlambda (ii)}
  \frac{q(p-2)\al^{p-q}}{ \la p(2-q)}\le 1\leq \frac{(p-2)\al^{p-q}}{ \la(2-q)}
\end{equation}
Using this estimate, $T_\la'(\al)=0$ and \eqref{Gl Tlambdap} we obtain
\begin{align*}
&\al^{q/2} T_\la''(\al) 
 = \frac{d}{d\al} \Big(\al^{q/2} T_\la'(\al) \Big) \\
&\hspace{.2cm}
 = \sqrt{\frac{pq}{8}}\frac{d}{d\al} \Big(\int_0^1 \frac{\la p(2-q)(1-s^q)-
      q(p-2)\al^{p-q}(1-s^p)}{\big(\la p(1-s^q)+ q\al^{p-q}(1-s^p)\big)^{3/2}}   \,ds \Big) \\
&\hspace{.2cm}
 = \sqrt{\frac{pq}{8}}\int_0^1 \frac{q(p-q)\al^{p-q-1}(1-s^p)
      \big( \la p(-2p+3q-2)(1-s^q)+ q(p-2)\al^{p-q} (1-s^p)\big)
      }{2\big(\la p(1-s^q)+ q\al^{p-q}(1-s^p)\big)^{5/2}} \,ds \\
&\hspace{.2cm}
 \le \sqrt{\frac{pq}{8}}\int_0^1 \frac{q(p-q)\al^{p-q-1}(1-s^p)
      \big( \la p(-2p+3q-2)(1-s^q)+ \la p(2-q)(1-s^p)\big)
      }{2\big(\la p(1-s^q)+ q\al^{p-q}(1-s^p)\big)^{5/2}} \,ds \\
&\hspace{.2cm}
 = \sqrt{\frac{pq}{8}}\int_0^1 \frac{\la pq(p-q)\al^{p-q-1}(1-s^p)
      \big(-2p+2q+(2p-3q+2)s^q - (2-q)s^p\big)
      }{2\big(\la p(1-s^q)+ q\al^{p-q}(1-s^p)\big)^{5/2}} \,ds.
\end{align*}
Since the function $s\mapsto (2p-3q+2)s^q - (2-q)s^p$ is increasing on $[0,1]$ and attains the value $2p-2q$ at $s=1$ we obtain $T_\la''(\al)<0$ .
\end{altproof}

\medskip

\begin{altproof}{Proposition \ref{Prop properties time map} (iii)}
The estimate for $\alpha_\lambda$ follows from \eqref{Gl Tlambda (ii)}. The lower estimates for $T_\la$ from (i) and the definition of $\al_{\la}$ from
Proposition~\ref{Prop properties time map}~(ii) moreover yield
\begin{align*}
T_{\la}(\al_{\la})
 &=  \max_{\al>0} T_\la(\al) 
  \ge \max_{\al>0}\Big(\frac{pq\al^{2-q}}{2p\la + 2q\al^{p-q}}\Big)^{1/2}
            \int_0^1 \frac{1}{\sqrt{1-s^p}}\ds \\
 &= \Big(\frac{q(p-2)}{2(p-q)}\Big)^{1/2}
      \Big(\frac{p(2-q)}{q(p-2)}\Big)^{(2-q)/(2(p-q))}
      \int_0^1  \frac{1}{\sqrt{1-s^p}}\ds
      \cdot \la^{(2-p)/2(p-q)}
\end{align*}
where the maximum is attained at $(\frac{\la p(2-q)}{q(p-2)})^{1/(p-q)}$. Similarly the upper estimate for $T_\la(\al_\la)$ is proved and we are done.

\end{altproof}

\medskip

\begin{altproof}{Proposition \ref{Prop properties time map} (iv)}
The formula for the time map from \eqref{Gl eq timemap} shows that the function $\la\mapsto T_\la(\al_\la)=\max_{\al>0} T_\la(\al)$ strictly decreases on $\R_{>0}$ from $+\infty$ to $0$. By the intermediate value theorem we deduce that there are uniquely determined positive numbers $\La_0<\La_1<\ldots<\Lambda_j\to\infty$ as ${j\to\infty}$ such that $T_{\Lambda_j}(\al_{\Lambda_j})=\frac{1}{2j+2}$ for all $j\in\N_0$. Moreover, the estimates from part (iii) give
\begin{align*}
\frac{1}{2j+2} = T_{\Lambda_j}(\al_{\Lambda_j}) \ge c_3{\Lambda_j}^{(2-p)/2(p-q)}, \qquad
\frac{1}{2j+2} = T_{\Lambda_j}(\al_{\Lambda_j}) \le C_3 {\Lambda_j}^{(2-p)/2(p-q)}
\end{align*}
which yields the estimates for $\Lambda_j$.
\end{altproof}

\medskip

\begin{altproof}{Proposition \ref{Prop properties time map} (v)}
For notational convenience set $c_\la:=(\tfrac{p|\la|}{q})^{1/(p-q)}$. The monotonicity of $T_\la$ on $(c_\la,\infty)$ follows from $\la<0$ and \eqref{Gl Tlambdap}. Hence, we obtain the result from
\begin{align*}
T_\la(c_\la)
 &= \int_0^{c_\la} \frac{1}{\sqrt{2(G_\la(c_\la)-G_\la(z))}} \,dz 
  = \int_0^{c_\la} \frac{1}{\sqrt{ \frac{2|\la|}{q}|z|^q - \frac{2}{p}|z|^p}} \,dz   \\
 &= \int_0^1 \frac{c_\la}{\sqrt{ \frac{2|\la|}{q}(c_\la t)^q - \frac{2}{p}(c_\la t)^p}}
      \,dt 
  = \sqrt{\frac{q}{2|\la|}} \cdot c_\la^{(2-q)/2}
     \int_0^1 \frac{1}{\sqrt{t^q -t^p}} \,dt   \\
 &= \Big( \frac{q}{2}\Big(\frac{p}{q}\Big)^{(2-q)/(p-q)}\Big)^{1/2} \int_0^1
      \frac{1}{\sqrt{t^q-t^p}}\,dt \cdot |\la|^{-(p-2)/2(p-q)} 
  = \frac{1}{2j+2} \Big(\frac{|\la_{j*}|}{|\la|}\Big)^{(p-2)/2(p-q)}
\end{align*}
where the latter equality follows from the definition of $\la_{j*}$, see
\eqref{Gl defn lambdajunten*}.
\end{altproof}

\medskip

In the proof of Proposition \ref{Prop Vorbereitung AprioriII 2} we use the following shorthand notation
$$
  m_q := \int_0^1 \frac{1}{\sqrt{1-s^q}}\ds,\qquad
  m_p := \int_0^1 \frac{1}{\sqrt{1-s^p}}\ds
$$

\medskip

\begin{altproof}{Proposition \ref{Prop Vorbereitung AprioriII 2} (i)}
From $0<\be_1,\be_2\le\al_\la\le C_3 \la^{1/(p-q)}$, see
Proposition~\ref{Prop properties time map} (iii), and
Proposition~\ref{Prop properties time map} (i) we get
\begin{align*}
T_\la(\be_1)
 &\le m_q \cdot \Big(\frac{pq \be_1^{2-q}}{2\la p+2q\be_1^{p-q}}\Big)^{1/2}
 \le m_q \cdot \Big(\frac{q}{2\la}\Big)^{1/2} \be_1^{(2-q)/2}, \\
  T_\la(\be_2)
 &\ge m_p \cdot
       \Big(\frac{pq\be_2^{2-q}}{2\la p+2q\be_2^{p-q}}\Big)^{1/2}
 \ge m_p\cdot
       \Big(\frac{pq}{2\la(p+qC_3^{p-q})}\Big)^{1/2} \be_2^{(2-q)/2}.
\end{align*}
From this we obtain
$$
\frac{T_\la(\be_1) }{T_\la(\be_2)  }
 \le \frac{m_q}{m_p}\Big(1+\frac{q}{p}C_3^{p-q}\Big)^{1/2}\Big(\frac{\be_1}{\be_2}\Big)^{(2-q)/2}
$$
or equivalently
\begin{equation} \label{Gl proof Vorbereitung AprioriII 2 1}
\frac{\be_2}{\be_1}
 \le \Big(\frac{m_q}{m_p}\Big(1+\frac{q}{p}C_3^{p-q}\Big)^{1/2}
      \frac{T_\la(\be_2)}{T_\la(\be_1)}\Big)^{2/(2-q)}.
\end{equation}
Hence, using the estimate from Proposition \ref{Prop properties time map} (i) and \eqref{Gl proof Vorbereitung AprioriII 2 1} we get
\begin{align*}
\frac{G_\la(\be_2)}{G_\la(\be_1)}
 &= \frac{\la p \be_2^q + q\be_2^p}{\la p\be_1^q + q\be_1^p} 
  = \frac{\be_2^q}{\be_1^q} \cdot \frac{2\la p+2q\be_2^{p-q}}{2\la p + 2q\be_1^{p-q}}\\
 &= \frac{\be_2^2}{\be_1^2}
     \cdot \frac{\frac{pq\be_1^{2-q}}{2\la p+2q\be_1^{p-q}}}
           {\frac{pq\be_2^{2-q}}{2\la p+2q\be_2^{p-q}}} 
  \le \frac{\be_2^2}{\be_1^2} \cdot \frac{m_q^2}{m_p^2}\Big(\frac{T_\la(\be_1)}{T_\la(\be_2)}\Big)^2\\
 &\le \Big(\frac{m_q}{m_p}\Big(1+\frac{q}{p}C_3^{p-q}\Big)^{1/2}
       \frac{T_\la(\be_2)}{T_\la(\be_1)}\Big)^{4/(2-q)}
        \cdot \frac{m_q^2}{m_p^2}\Big(\frac{T_\la(\be_1)}{T_\la(\be_2)}\Big)^2 \\
 &= \Big(\frac{m_q}{m_p}\Big)^{(8-2q)/(2-q)} \Big(1+\frac{q}{p}C_3^{p-q}\Big)^{2/(2-q)}
      \Big(\frac{T_\la(\be_2)}{T_\la(\be_1)}\Big)^{2q/(2-q)}.
\end{align*}
\end{altproof}

\medskip

\begin{altproof}{Proposition \ref{Prop Vorbereitung AprioriII 2} (ii)}
From $\be_1,\be_2\ge \al_\la\ge c_3\la^{1/(p-q)}>0$, see
Proposition~\ref{Prop properties time map} (iii), we get
\begin{align*}
T_\la(\be_2)
 &\le m_q \cdot
      \Big(\frac{pq\be_2^{2-p}}{2p \la \be_2^{q-p}+2q}\Big)^{1/2}
  m_q \cdot
      \Big(\frac{p }{2}\Big)^{1/2}\be_2^{(2-p)/2}, \\
T_\la(\be_1)
 &\ge m_p \cdot
       \Big(\frac{pq\be_1^{2-p}}{2p \la \be_1^{q-p}+2q}\Big)^{1/2}
 \ge m_p\cdot
       \Big(\frac{pq}{2p c_3^{q-p}+2q}\Big)^{1/2} \be_1^{(2-p)/2}
\end{align*}
hence
$$
\frac{\be_2^2}{\be_1^2}
 \le \Big( \frac{m_q}{m_p}\Big(1+\frac{p}{q}c_3^{q-p}\Big)^{1/2}
      \frac{T_\la(\be_1)}{T_\la(\be_2)} \Big)^{4/(p-2)}.
$$
As before this implies
\begin{align*}
\frac{G_\la(\be_2)}{G_\la(\be_1)}
 &= \frac{\be_2^2}{\be_1^2} \cdot
     \frac{\frac{pq\be_1^{2-q}}{2\la p+2q\be_1^{p-q}}}
          {\frac{pq\be_2^{2-q}}{2\la p+2q\be_2^{p-q}}} \\
 &\le \Big(\frac{m_q}{m_p}\Big(1+\frac{p}{q}c_3^{q-p}\Big)^{1/2}
        \frac{T_\la(\be_1)}{T_\la(\be_2)}\Big)^{4/(p-2)} \cdot
        \frac{m_q^2}{m_p^2}\Big(\frac{T_\la(\be_1)}{T_\la(\be_2)}\Big)^2 \\
 &= \Big(\frac{m_q}{m_p}\Big)^{2p/(p-2)} \Big(1+\frac{p}{q}c_3^{q-p}\Big)^{2/(p-2)}
      \Big(\frac{T_\la(\be_1)}{T_\la(\be_2)}\Big)^{2p/(p-2)}.
\end{align*}
\end{altproof}

\medskip

\begin{altproof}{Proposition \ref{Prop Vorbereitung AprioriII 2} (iii)}
The estimate for $\be_1\le \al_\la\le \be_2$ from the assertion follows from
$$
\frac{G_\la(\be_1)}{G_\la(\be_2)}
 = \frac{G_\la(\be_1)}{G_\la(\al_\la)} \cdot
     \frac{G_\la(\al_\la)}{G_\la(\be_2)}
$$
and the inequalities which we have proved in part (i) and part (ii).
\end{altproof}

\begin{ack}
  The second author expresses his gratitude to the Klaus-Tschira-Stiftung for providing financial support while this project was accomplished. The authors would like to thank Professor Wolfgang Reichel and Professor Michael Plum for several helpful discussions about the subject.
\end{ack}

\bibliographystyle{alpha}

\end{document}